\documentclass[%
a4paper
]{mpi2015-cscpreprint}

\usepackage[american]{babel}

\usepackage{graphicx}

\usepackage{amssymb}
\usepackage{amsthm}
\usepackage{mymacros}
\usepackage{array}

\usepackage{mathptmx}
\usepackage{amsmath}
\usepackage{amsfonts}
\usepackage{booktabs}

\usepackage{placeins}
\usepackage{multirow}

\usepackage{algorithm}
\usepackage{algorithmicx}
\usepackage{algpseudocode}

\usepackage{subcaption}
 \usepackage{numprint}

\usepackage{pgfplots}
\pgfplotsset{compat=1.13}
\usetikzlibrary{shapes.multipart,decorations.markings,decorations.text, 
decorations.pathreplacing,patterns,positioning,intersections,calc,matrix}

\newtheorem{definition}{Definition}[section]
\newtheorem{theorem}[definition]{Theorem}
\newtheorem{lemma}[definition]{Lemma}

\newtheorem{corollary}[definition]{Corollary}
\newtheorem{remark}[definition]{Remark}

\begin{document}
  

\title{Stable and Efficient Computation of Generalized Polar Decompositions}
  
\author[$\dagger$]{Peter Benner}
\affil[$\dagger$]{Max Planck Institute for Dynamics of Complex Technical Systems, \authorcr
    Sandtorstr. 1, 39106 Magdeburg, Germany.}
  
\author[$\ddagger$]{Yuji Nakatsukasa}
\affil[$\ddagger$]{Mathematical  Institute,  University  of  Oxford,  \authorcr
Andrew Wiles Building,
Woodstock Road, Oxford,  OX2 6GG, UK.}
  
\author[$\dagger\ast$]{Carolin Penke}
\affil[$\ast$]{Corresponding author.  \email{penke@mpi-magdeburg.mpg.de}}
  
\shorttitle{New Algorithms for Bethe-Salpeter Eigenvalue problems}
\shortauthor{P. Benner, C.Penke}
  

  
\abstract{%
We present methods for computing the generalized polar decomposition of a matrix based on the dynamically weighted Halley (DWH) iteration. This method is well established for computing the standard polar decomposition. A stable implementation is available, where matrix inversion is avoided and QR decompositions are used instead. We establish a natural generalization of this approach for computing generalized polar decompositions with respect to signature matrices. Again the inverse can be avoided by using a generalized QR decomposition called hyperbolic QR decomposition. However, this decomposition does not show the same favorable stability properties as its orthogonal counterpart. We overcome the numerical difficulties by generalizing the CholeskyQR2 method. This method computes the standard QR decomposition in a stable way via two successive Cholesky factorizations. An even better numerical stability is achieved by employing permuted graph bases, yielding residuals of order $10^{-14}$ even for badly conditioned matrices, where other methods fail. }

\novelty{We provide practical iterations for computing generalized polar decompositions and follow up on two new ideas to improve the stability of the iteration.
\begin{enumerate}
\item Using the hyperbolic QR decomposition and the $LDL^T$ factorization. 
\item Using a well-conditioned subspace basis by exploiting a connection to (Lagrangian) graph subspaces.
\end{enumerate}
}

\maketitle

 \section{Introduction}
\label{Sec:Introduction}
For $\mathbb{K}=\mathbb{C}$ or $\mathbb{K}=\mathbb{R}$, the polar decomposition of a matrix $A\in\mathbb{K}^{m\times n}$, $m\geq n$, is given as 
\begin{align}\label{Eq:StandPolDec}
A=UH,\quad U^{*}U = I,\quad H=H^{*} \geq0
\end{align}
where $U\in\mathbb{K}^{m\times n}$ is unitary and $H\in\mathbb{K}^{n\times n}$ is positive semidefinite. $\cdot^*$ is a placeholder for the transpose $\cdot^{\tran}$ or Hermitian transpose $\cdot^{\herm}$ depending on the considered field. It is a well-known tool in numerical linear algebra, intimately connected to the singular value decomposition (SVD). While it can be regarded as a ``tuned down'' variant of the SVD, it is worth to be studied in its own right. The decomposition is of use in many applications, in particular because of its best-approximation properties. For a detailed treatment see Chapter 8 in \cite{Hig08}.  

Classically, the SVD is the starting point for the computation of the polar decomposition \eqref{Eq:StandPolDec}. This method is not very pleasing from an algorithmic point of view, as the polar decomposition contains less (but still very useful) information than the SVD. This route therefore computes more than might be necessary in a given application. In recent years, methods have been developed to compute the polar decomposition efficiently on modern computer architectures \cite{LtaSEetal19,NakBG10,NakF16,NakH12}. In fact, the polar decomposition can now be seen as a first step towards computing the SVD of a general matrix \cite{SukLEetal19}. Efficient algorithms for computing the SVD of large matrices on high performance architectures form an active field of research.

It is well known that the unitary polar factor of a Hermitian matrix coincides with the matrix sign function \cite{Hig08}. The matrix sign function is a widely used tool for acquiring invariant subspaces of a matrix. This property is used to solve matrix equations \cite{BenQ99,Rob80} and develop parallelizable algorithms for solving eigenvalue problems \cite{BaiD93,SunQ04}. Therefore, efficient iterations for computing the polar decomposition, such as the QDWH iteration \cite{NakH12} and its successor based on Zolotarev's functions \cite{NakF16}, can be used to improve these methods for Hermitian matrices. 

The concept of polar decompositions can be generalized in terms of non-standard inner product spaces. The papers \cite{BolR95, BolVRetal97, MehRR06} treat inner products induced by Hermitian matrices, while \cite{HigMMetal05,HigMT10} provide a more general treatment. Let $A\in\mathbb{K}^{m\times n}$, and $M\in\mathbb{K}^{m\times m}$, $N\in\mathbb{K}^{n\times n}$ be nonsingular. Under certain assumptions on $A$, $M$ and $N$ (see \cite{HigMT10}), $A$ has a \emph{(canonical) generalized polar decomposition} with respect to the inner products induced by $M$ and $N$:
\begin{align}\label{Eq:GenPolDec}
 A = WS,
\end{align}
where $W\in\mathbb{K}^{m\times n}$ is a partial $(M,N)$-isometry. $S\in\mathbb{K}^{n\times n}$ is self-adjoint with respect to $N$ and its nonzero eigenvalues are contained in the open right half plane.

The standard polar decomposition \eqref{Eq:StandPolDec} can be used to solve the orthogonal Procrustes problem, arising in fields such as marketing in the context of multidimensional scaling \cite{BoGn05}.  A generalized polar decomposition can be used as a tool to solve the non-orthogonal variant \cite{Kin05}.

In analogy with the standard setting, the factor $W$ of the generalized polar decomposition \eqref{Eq:GenPolDec} coincides with the matrix sign function of a square matrix $A$ if $A$ is self-adjoint with respect to the defining inner product. This is shown in Section \ref{Sec:Prel} of this paper. Finding efficient iterations for computing the generalized polar decomposition can therefore lead to new methods for matrix equations and eigenvalue problems involving self-adjoint matrices.

In this work, we present some results on how generalized polar decompositions can be computed based on the dynamically weighted Halley (DWH) iteration. This iteration is successful in computing the standard polar decomposition in an efficient and stable way \cite{NakH12}. We focus on the important subclass of inner products induced by signature matrices, i.e.\ diagonal matrices with $+1$ and $-1$ as diagonal values, denoted by $\Sigma$ throughout the paper.
Self-adjoint matrices with respect to $\Sigma$ are called pseudosymmetric. They show up in the field of computational quantum physics \cite{Cas95,OniRR02}, from which our main motivation is drawn. Ab initio simulations aim to predict properties of materials from first principles. Discretizations of the underlying differential equations often lead to structured eigenvalue problems of very large size. Consider, e.g., the Bethe-Salpeter eigenvalue problem. The eigenvalues and eigenvectors of a block matrix
\begin{eqnarray*}
 H_{BS} = \begin{bmatrix}
           A& B\\
	   -\bar{B}& -\bar{A}
          \end{bmatrix} = 
          \begin{bmatrix}
           A & B \\
           -B^{\herm} & -A^{\tran}
          \end{bmatrix},\          
          \quad A=A^{\herm},\quad B=B^{\tran} \in\mathbb{C}^{n\times n},\nonumber
\end{eqnarray*}
are used to determine optical properties of crystalline systems \cite{ShaDYetal16} or molecules \cite{BlaDJetal20}.  
$H_{BS}$ has the additional property, coming from physical constraints of the original problem, that $\Sigma H_{BS}$ is positive definite for $\Sigma=\diag{I_n,-I_n}$. Similar structures arise in different contexts of electronic structure theory \cite{BenKK16, DonGKetal84,MehMX04}. 
We call pseudosymmetric matrices with this property definite pseudosymmetric matrices. For these matrices in particular, the convergence behaviour of our proposed method will turn out to be as good as in the standard setting defined by the Euclidean inner product.
Pseudosymmetric matrices also play a role in describing damped oscillations of linear systems. See \cite{Ves11}, where they are called $J$-Hermitian and definite pseudosymmetric matrices are called $J$-positive.

The remainder of this paper is structured as follows. Section \ref{Sec:Prel} fixes the notation on inner products and related aspects which form basic concepts used throughout the remaining paper. 
Section \ref{Sec:LDLHQR} clarifies how generalizations of the QR factorization can be used to compute matrices that are orthogonal with respect to non-standard inner products. 
In Section \ref{Sec:QDWH}, we recapitulate the central ideas of the QDWH algorithm.
Section \ref{Sec:GenPolDec} shows how they can be applied in order to compute a generalized polar decomposition. 
We show general results and then restrict ourselves to inner products induced by signature matrices. Here, inverses can be avoided by using the decompositions presented earlier in Section \ref{Sec:LDLHQR}. The introduction of permuted graph bases can improve the stability of the computation of the generalized polar factor. Details are found in Section \ref{Sec:PermGraph}.  Section \ref{Sec:NumRes} gives numerical results on the questions of stability and convergence. Conclusions and further research directions are given in Section \ref{Sec:Conclusions}.

\section{Preliminaries}\label{Sec:Prel}
Following \cite{HigMMetal05} and \cite{MacMT05}, we provide basic notation regarding inner products needed for the generalized polar decomposition. A nonsingular matrix $M$ defines an \emph{inner product} on $\mathbb{K}^{n}$ ($\mathbb{K}\in\{\mathbb{C},\mathbb{R}\}$), which is a bilinear or sesquilinear form $\langle.,.\rangle_M$,  given by
\begin{align*}
\langle x, y \rangle_M = \begin{cases}
                          x^{\tran} My\text{ for bilinear forms,}\\
                          x^H My\text{ for sesquilinear forms,}
                         \end{cases} 
\end{align*}
for $x,y\in\mathbb{K}^{n}$. We use $\cdot^*$ throughout the paper to indicate transposition 
$\cdot^{\tran}$
or conjugated transposition 
$\cdot^{\herm}$
, depending on whether a bilinear or sesquilinear form is given. We overline a quantity to denote complex conjugation.

For a matrix $A\in\mathbb{K}^{m\times n}$,  $A^{\star_{M,N}}\in\mathbb{K}^{n\times m}$ denotes the adjoint with respect to the inner products defined by the nonsingular matrices $M\in\mathbb{K}^{m\times m}$, $N\in\mathbb{K}^{n\times n}$. This matrix is uniquely defined by satisfying the identity
\begin{align*}
 \langle Ax, y \rangle_M = \langle x, A^{\star_{M,N}}y \rangle_N 
\end{align*}
for all $x\in\mathbb{K}^{n},y\in\mathbb{K}^{m}$. We call $A^{\star_{M,N}}$ the \emph{$(M,N)$-adjoint} of $A$ and it holds
\begin{align}\label{Eq:StarMN}
 A^{\star_{M,N}} = N^{-1}A^*M.
\end{align}
$A$ is \emph{$(M,N)$-orthogonal} if $A^{\star_{M,N}}A = I_n$.
This notion is generalized in the form of partial $(M,N)$-isometries. A matrix $A$ is called a \emph{partial $(M,N)$-isometry} when $ AA^{\star_{M,N}}A = A$.

If $A$ is square and $M=N$, the notation simplifies. The \emph{$M$-adjoint} is given by  $A^{\star_M}= A^{\star_{M,M}}$. We call a square matrix $A$ an \emph{($M$-)automorphism} if $A^{\star_M} = A^{-1}$ (given the inverse exists), and  \emph{($M$-)self-adjoint} if $A=A^{\star_M}$.

In the following we give basic results regarding the generalized polar decomposition \eqref{Eq:GenPolDec}. They can be found in \cite{HigMMetal05} or \cite{HigMT10}. For certain matrices $M\in \mathbb{K}^{m\times m}$, $N\in\mathbb{K}^{n\times n}$, the canonical generalized polar decomposition can be defined. $M$ and $N$ are required to form an \emph{orthosymmetric pair}, i.e.\  it must hold 
\begin{itemize}
 \item[(a)] $M^{\tran} = \beta M$, $N^{\tran}=\beta N$, $\beta=\pm 1$ for bilinear forms,
 \item[(b)] $M^{\herm} = \alpha M$, $N^{\herm}=\alpha N$, $|\alpha|= 1$ for sesquilinear forms.
\end{itemize}

\begin{definition}[Definition 3.6 in \cite{HigMT10}]\label{Def:GenPolDec}
 A matrix $A\in\mathbb{K}^{m\times n}$ has a \emph{canonical generalized polar decomposition} with respect to an orthosymmetric pair of matrices $M\in\mathbb{K}^{m\times m}$ and $N\in\mathbb{K}^{n\times n}$, if there exists a partial $(M,N)$-isometry $W$ and an $N$-self-adjoint matrix $S$, whose eigenvalues all have positive real parts, s.t.\ 
  \begin{align*}
   A=WS,
  \end{align*}
  and $\range{W^{\star_{M,N}}}=\range{S}$.
\end{definition}
If $A$ has full column rank, $W$ is $(M,N)$-orthogonal. If additionally $A$ is square and $M=N$, then $W$ is an  $N$-automorphism.

In contrast to the standard polar decomposition, the existence of the (canonical) generalized polar decomposition can in general not be guaranteed. The following theorem clarifies this issue.

\begin{theorem}[Existence of the canonical generalized polar decomposition, Theorem 3.9 in \cite{HigMT10}]
A matrix $A\in\mathbb{K}^{m\times n}$ has a unique canonical generalized polar decomposition with respect to the orthosymmetric pair $M$, $N$ if and only if 
\begin{enumerate}
 \item $A^{\star_M}A$ has no eigenvalues on the negative real axis,
 \item if zero is an eigenvalue of $A^{\star_{M,N}}A$, then it is semisimple and
\item $\kernel{A^{\star_{M,N}}}=\kernel{A}$.
\end{enumerate}
\end{theorem}

In case of existence it holds $S={(A^{*_{M,N}}A)}^{\frac{1}{2}}$, and $W^{\star_{M,N}}WS=S$. Just as the standard polar decomposition, the generalized polar decomposition is related to the matrix sign function. This is a generalization of the scalar sign function
\begin{align*}
\sign{z} = 
 \begin{cases}
1,&\ \Real{z} >0,\\
-1,&\ \Real{z} <0,  
 \end{cases}
 \qquad z \in\mathbb{C},\ z\notin  \iu\mathbb{R},
\end{align*}
applied to matrices. For a detailed treatment see \cite{Hig08}, Chapter 5. Let a square matrix $A$ without purely imaginary eigenvalues have a Jordan decomposition $A=Z\diag{J_+,J_-}Z^{-1}$, where $J_+\in\mathbb{K}^{n_+\times n_+}$ contains Jordan blocks associated with eigenvalues with positive real part and $J_-\in\mathbb{K}^{n_-\times n_-}$ contains Jordan blocks associated with eigenvalues with negative real part. Then the matrix sign function is defined as 
\begin{align*}
 \sign{A} := Z\diag{I_{n_+},-I_{n_-}}Z^{-1}.
\end{align*}

\begin{theorem}
 Let $M$ be a nonsingular matrix and $A\in\mathbb{K}^{n\times n}$ be self-adjoint with respect to the inner product induced by $M$. If $A$ has no purely imaginary eigenvalues, $\sign{A}$ and the canonical generalized polar decomposition (with respect to $M$) $A=WS$ are well-defined and it holds
 \begin{align*}
  \sign{A} = W.
 \end{align*}
\end{theorem}
\begin{proof}
 The matrix sign function can be expressed as \cite{Hig08}
 \begin{align*}
  \sign{A} = A(A^2)^{-1/2}.
 \end{align*}
 The generalized polar decomposition $A=WS$ is well-defined with a unique self-adjoint factor $S$ if $M^{-1}A^*MA$ has no negative real eigenvalues.  For self-adjoint matrices it holds $M^{-1}A^*M=A$, so $M^{-1}A^*MA = A^2$ can only have negative real eigenvalues if $A$ has purely imaginary eigenvalues. So $A=WS$ is well-defined. Using $S=(M^{-1}A^*MA)^{\frac{1}{2}}$, $W$ can be given as
 \begin{align*}
  W = A(M^{-1}A^*MA)^{-1/2} = A(A^2)^{-1/2} = \sign{A}.
 \end{align*}
\end{proof}

\section{Connections between $LDL^{\tran}$ factorizations and hyperbolic QR decompositions}\label{Sec:LDLHQR}
\subsection{The hyperbolic QR factorization}
 A matrix $\Sigma=\diag{\sigma_1,\dots,\sigma_n}$, where $\sigma_i\in\{+1,-1\}$ for $i=1,\dots,n$, is called a \emph{signature matrix}. We search for a way to compute $(\Sigma,\hat{\Sigma})$-orthogonal bases, which span a given subspace. While $\Sigma$ is a given signature matrix, $\hat{\Sigma}$ can be another arbitrary signature matrix. $(\Sigma,\hat{\Sigma})$-orthogonal matrices are also called hyperexchange matrices \cite{Hig03} and can be used to solve indefinite least square problems \cite{BojHP03}.

The methods presented in this section take a rectangular matrix $A\in\mathbb{K}^{m\times n}$ and signature matrix $\Sigma$ as input and deliver two outputs. These are another signature matrix $\hat{\Sigma}$, and $H\in\mathbb{K}^{m\times n}$, which spans the same subspace as $A$ and is $(\Sigma,\hat{\Sigma})$-orthogonal. Subspace representations of this kind will be used in the computation of generalized polar decompositions (Section \ref{Sec:GenPolDec}). 
A classic method for computing such a subspace representation uses the hyperbolic  QR decomposition.
\begin{theorem}[The hyperbolic QR decomposition \cite{BunB85}]\label{Thm:HypQR}
 Let $\Sigma\in\mathbb{R}^{m\times m}$ be a signature matrix, $A\in\mathbb{K}^{m\times n},\ m\geq n$. Suppose all the leading principal submatrices of $A^*\Sigma A$ are nonsingular. Then there exists a permutation $P$, a signature matrix $\hat{\Sigma}=P^{\tran}\Sigma P$, a $(\Sigma,\hat{\Sigma})$-orthogonal matrix $H\in\mathbb{K}^{m\times m}$ (i.e.\ $H^*\Sigma H= \hat{\Sigma}$), and an upper triangular matrix $R\in\mathbb{R}^{n\times n}$, such that
 \begin{align*}
  A=H\begin{bmatrix}
R\\
0
     \end{bmatrix}.
 \end{align*}
\end{theorem}

The hyperbolic QR decomposition is unique when the diagonal values of $R$ are restricted to be positive real \cite{SinS00}. 
\begin{remark}
 The hyperbolic QR decomposition can be truncated to form a thin hyperbolic QR decomposition
 \begin{align*}
  A = H_0R,\ H_0\in\mathbb{K}^{m\times n},\ R\in\mathbb{K}^{n\times n},\ H_0^*\Sigma H_0 = \hat{\Sigma}_0.
 \end{align*}
 $H_0$ contains the first $n$ columns of $H$  and $\hat{\Sigma}_0$ contains the $n\times n$ leading submatrix of $\hat{\Sigma}$, where $H$ and $\hat{\Sigma}$ are given in Theorem \ref{Thm:HypQR}.
\end{remark}

The hyperbolic QR decomposition can be computed by accumulating transformations that introduce zeros below the diagonal, similar to the standard QR decomposition. We give a quick idea on how these elimination matrices are computed. For a more formal treatment, see e.g. \cite{Wat07}.
For a given vector $x$ and a given signature matrix $\Sigma$, we look for a transformation $H$ such that $H^{-1}x=d e_1$, where $e_1$ denotes the first unit vector and $H^ {\herm}\Sigma H = \hat{\Sigma}$ is another signature matrix.   
The two kinds of transformations used are orthogonal Householder transformations and hyperbolic Givens rotations. For illustrative purposes suppose $x\in\mathbb{C}^{2n}$ and $\Sigma=\diag{I_n,-I_n}$. Let
\begin{align*} 
H_1=\begin{bmatrix}                                                                                                                                                                                             H_+&\\
&H_-\end{bmatrix},
\end{align*}
where $H_+$ and $H_-$ are Householder transformations of dimension $n\times n$, such that $H_1^{-1}x = a e_1 + be_{n+1}$. We have $H_1^ {\herm}\Sigma H_1=\Sigma$. The $b$ entry in position $n+1$ is annihilated by a hyperbolic Givens rotation acting on row $1$ and $n+1$. $G^{-1}
\begin{bmatrix}
 a\\ b
\end{bmatrix}=
\begin{bmatrix}
 d\\0
\end{bmatrix}$                                                                                                                                                                                                                                                                                          
 is achieved by 
\begin{subequations}
\label{eq:Hgivens}
\begin{gather*}
G^{-1}=\begin{bmatrix}
    c & -s\\
    -\overline{s} & c
   \end{bmatrix},\\ \text{where }
   \begin{cases}
c=|a|/\sqrt{|a|^2 - |b|^2},\ s =e^{i\phi} |b|/\sqrt{|a|^2 - |b|^2}\qquad &\text{if $|a|>|b|$,}\\
c=|a|/\sqrt{|b|^2 - |a|^2},\ s =e^{i\phi}  |b|/\sqrt{|b|^2 - |a|^2}\qquad &\text{if $|a|<|b|$,}
   \end{cases}
\end{gather*}
\end{subequations}
with $\phi=\arg{a}-\arg{b}$.

   $G$ is given as $G=\begin{bmatrix}
    c & \overline{s}\\
    s & c
   \end{bmatrix}$.
For the $|a|>|b|$ case we have 
\begin{align} \label{Eq:G1}
G^ {\herm}\begin{bmatrix}
                                1&0\\0&-1
                               \end{bmatrix}G=\begin{bmatrix}
                                1&0\\0&-1
                               \end{bmatrix}. 
\end{align}
                               For $|a|<|b|$ there is a sign switch in the signature matrix, 
                         
\begin{align} \label{Eq:G2}
G^ {\herm}\begin{bmatrix}
                                1&0\\0&-1
                               \end{bmatrix}G=\begin{bmatrix}
                                -1&0\\0&1
                               \end{bmatrix}.
\end{align}
If $a$ and $b$ are real then $G$ is also real. Embedding $G$ into a larger matrix $H_2$ (equal to the identity except in rows and columns $1$ and $n+1$), gives the sought-after transformation $H=H_1H_2$. $H^ {\herm}\Sigma H=\hat{\Sigma}$ is another signature matrices, in which $+1$ at diagonal position 1 and $-1$ at diagonal position $n+1$ have been interchanged, if \eqref{Eq:G2} takes effect. If \eqref{Eq:G1} takes effect, the signature matrix does not change: $\Sigma=\hat{\Sigma}$. 

The presented method works not only for the specific signature matrix $\Sigma=\diag{I_n,-I_n}$.
For an arbitrary signature matrix $\Sigma$, $H_+$ acts on the rows corresponding to positive entries of $\Sigma$, $H_-$ acts on the remaining rows.  $H_1$ is set up accordingly. $H_2$ then acts on the remaining two entries and may or may not introduce a sign switch in the signature matrix. In \eqref{eq:Hgivens}, the case $|a|=|b|$ is not covered and in this case no suitable matrix $G$ exists. The assumptions in Theorem \ref{Thm:HypQR} prevent this from happening. However, if $a$ and $b$ are close, $G$ becomes ill-conditioned. This can lead to an instability in algorithms employing this kind of column elimination.

In order to overcome these potential instabilities, we once again take a look at the standard QR decomposition. Here, we can find Cholesky-QR as an alternative computational approach, explained below. It has been rarely considered because its unmodified variant is less stable than the classical approach using Householder transformations.

The orthogonal QR decomposition is connected to a Cholesky factorization in the following way \cite{YamNYF15}.  If $A=QR$ is a QR decomposition, then $A^*A=R^*R$ is a Cholesky factorization of $A^*A$. Conversely, if the Cholesky factorization $A^*A=R^*R$ with nonsingular $R$ is given, $Q=AR^{-1}$ is the orthogonal factor of the QR decomposition.

In the indefinite setting an analogous connection exists between the hyperbolic QR factorization (Theorem \ref{Thm:HypQR}) and a scaled variant of the $LDL^{\tran}$ factorization given in \cite[Thm.~4.1.3]{GolV13}.

\begin{lemma}\label{Thm:HermIndFac1}
 Let $A\in\mathbb{K}^{m\times n}$ have a decomposition $A=HR$, $H\in\mathbb{K}^{m\times n}$, $R\in\mathbb{K}^{n\times n}$. Then
 \begin{gather}\label{eq:FundDecom}
 H^*\Sigma H=\hat{\Sigma} \qquad  \Leftrightarrow \qquad A^*\Sigma A = R^* \hat{\Sigma} R.  
 \end{gather}
\end{lemma}

\begin{remark}\label{Rem:HRviaLDL}
If the right side of the equivalence in \eqref{eq:FundDecom} with nonsingular $R$ is given, $H=AR^{-1}$  can be recovered from $A$ and $R$. In the case of signature matrices, the right side can be computed from an $LDL^ {\tran}$ decomposition
$A^*\Sigma A = LDL^*$, where $L$ is unit lower triangular, $D$ is real diagonal. Then $R:=|D|^\frac{1}{2}L^*$ and $\hat{\Sigma}:=\sign{D}$ (containing the signs of the diagonal values in $D$) fulfill $A^*\Sigma A = R^* \hat{\Sigma} R$.
\end{remark}

The $LDL^{\tran}$ factorization with a strictly diagonal $D$ is typically not used in modern algorithms, as it becomes unstable when small diagonal values appear \cite{AshGL99}. Instead, $D$ is allowed to be block-diagonal with $1\times 1$ and $2\times 2$ blocks, and a pivoting scheme is employed \cite{BunP71}. This yields an $LDL^{\tran}$ factorization  $A=PLDL^*P^{\tran}$, where $P$ is a permutation matrix, $L$ is unit lower triangular and $D$ is block-diagonal. We call this factorization ``$LDL^ {\tran}$ factorization with pivoting'' or ``block $LDL^ {\tran}$ factorization'' in order to distinguish it from the ``diagonal $LDL^ {\tran}$ factorization''. The additional degrees of freedom destroy the uniqueness property, but allow for a more stable computation. Several backward stable algorithms have been developed (see \cite{BunK77,BunKP76}) and well-established implementations are available in software packages such as LAPACK and MATLAB \cite{AshGL99,Duf04}. In the latter, the implementation is given as the \texttt{ldl} command. 

Remark \ref{Rem:HRviaLDL} points out how the hyperbolic QR decomposition can be computed from the diagonal $LDL^ {\tran}$ decomposition. If instead the $LDL^ {\tran}$ decomposition with pivoting is used, one obtains the (thin) indefinite QR factorization, which is not unique anymore.

\begin{theorem}[(Thin) indefinite QR factorization \cite{SinS00}]\label{Thm:IQR}
 Let $\Sigma\in\mathbb{K}^{m\times m}$ be a signature matrix, $A\in\mathbb{K}^{m\times n},\ m\geq n$. Suppose $A^*\Sigma A$ is nonsingular. Then there exists a factorization
 \begin{align*}
  A=HRP^{\tran},\qquad H\in\mathbb{K}^{m\times n},\ R\in\mathbb{K}^{n\times n},\ P\in\mathbb{R}^{n\times n}.
 \end{align*}
 $P$ is a permutation matrix. $P_\Sigma\in\mathbb{R}^{m\times n}$ contains $n$ columns of an $m\times m$ permutation matrix and defines the signature matrix $\hat{\Sigma}=P_\Sigma^{\tran}\Sigma P_\Sigma$. $H$ is $(\Sigma,\hat{\Sigma})$-orthogonal (i.e.\ $H^*\Sigma H= \hat{\Sigma}$), and $R$ is block-upper triangular with blocks of size $1\times 1$ and $2\times 2$.
\end{theorem}

The difference between the indefinite QR factorization (Theorem \ref{Thm:IQR}) and the hyperbolic QR factorization (Theorem \ref{Thm:HypQR}) is that pivoting is introduced, which results in the second permutation matrix $P$. $2\times 2$ blocks appear on the diagonal of $R$, and the assumption on $A^*\Sigma A$ is weaker. \cite{Sin06} describes how this decomposition can be computed via the successive use of transformation matrices, similar to the hyperbolic QR decomposition (Theorem \ref{Thm:HypQR}). 

A perturbation analysis for the computation of the hyperbolic QR factorization (Theorem \ref{Thm:HypQR}), i.e.\ the triangular case of the indefinite QR factorization in Theorem \ref{Thm:IQR}, is given in \cite{SinS00} and more recently in \cite{LiYS12}. 

Computing the Indefinite QR factorization via the $LDL^ {\tran}$ factorization proceeds as follows.
\begin{enumerate}
 \item Compute an $LDL^ {\tran}$ factorization $A^*\Sigma A = PLDL^*P^ {\tran}$, where $D$ is block-diagonal.
 \item Diagonalize $D$, i.e., compute unitary $V$, diagonal $\Lambda$, s.t. $V\Lambda V^* = D$. $V$ has the same block-diagonal structure as $D$.
 \item Set $R={|\Lambda|}^\frac{1}{2}V^*L^*$,  $H=APR^{-1}$.
\end{enumerate}

\subsection{LDLIQR2: Computing the indefinite QR factorization via two $LDL^{\tran}$ decompositions}\label{Sec:LDLHQR2} 

In \cite{YamNYF15},  the CholeskyQR2 algorithm is formulated and following these ideas we derive the indefinite variant (see also \cite{BenP20a}). We call the algorithm LDLIQR2, standing for $\textbf{LDL}^{\tran}$-based computation of the \textbf{I}ndefinte \textbf{QR} decomposition, applied \textbf{twice}. It computes a $(\Sigma,\hat{\Sigma})$-orthogonal basis of the subspace spanned by a matrix $A$. $\Sigma$ is a given signature matrix and $\hat{\Sigma}$ is another signature matrix determined by the algorithm.   It starts by computing the indefinite QR factorization $A=H_1R_1P_1^{\tran}$ via the $LDL^{\tran}$ factorization with pivoting as described in the previous section. Then as a second step, the indefinite QR decomposition $H_1=HR_2P_2^{\tran}$ is computed using the same method. This yields a factorization 
\begin{equation}
\begin{aligned}\label{Eq:IQR2}
A=HR_2P_2^{\tran}R_1P_1^{\tran}, \text{ with } &R_1,\ R_2\text{ upper triangular, }\\ &P_1,\ P_2\text{ permutation matrices.} 
\end{aligned}
\end{equation}

In exact arithmetic, the second step is redundant, as the hyperbolic QR decomposition of a $(\Sigma,\hat{\Sigma})$-orthogonal $H$ is $H=HI$. In floating point arithmetic, however, we hope to see improvements regarding the accuracy of the computed factorization. $P_2$ will in practice often be the identity matrix. In this case, we have computed an instance of the Indefinite QR factorization given in Theorem \ref{Thm:IQR} with $R:=R_2R_1$, $P:=P_1$. For our application we are just interested in a $(\Sigma,\hat{\Sigma})$-orthogonal basis, so the exact shape of $R$ in a decomposition $A=HR$ does not matter. The method is formulated in Algorithm \ref{Alg:LDLIQR2}.

\begin{algorithm}[ht]
\caption{LDLIQR2: Compute $(\Sigma,\hat{\Sigma})$-orthogonal basis via double $LDL^{\tran}$ factorization with pivoting.\label{Alg:LDLIQR2}}
\label{dgeqrt}
\begin{algorithmic}[1]
\Require $A\in\mathbb{K}^{m\times n}$, with full column rank, $\Sigma\in\mathbb{R}^{n\times n}$ is a signature matrix.
\Ensure $(\Sigma,\hat{\Sigma})$-orthogonal $H\in\mathbb{K}^{m\times n}$  and $R_1$, $P_1$, $R_2$, $P_2\in\mathbb{K}^{n\times n}$ as in \eqref{Eq:IQR2}.
\Statex \textbf{// First pass:}
\State $[L_1,D_1,P_1]\leftarrow \texttt{ldl}(A^{*}\Sigma A)$
\State $[V_1,\Lambda_1]\leftarrow \texttt{eig}(D)$ \Comment{$V_1$ is block-diagonal.}
\State $H_1 \leftarrow AP_1 L_1^{-*}V_1{|\Lambda_1|}^{-\frac{1}{2}}$
\State $R_1 \leftarrow {|\Lambda_1|}^\frac{1}{2} V_1^* L_1^*$
\Statex \textbf{// Second pass:}
\State $[L_2,D_2,P_2]\leftarrow \texttt{ldl}(H^{*}\Sigma H)$
\State $[V_2,\Lambda_2]\leftarrow \texttt{eig}(D)$ \Comment{$V_2$ is block-diagonal.}
\State $H \leftarrow H_1P_2 L_2^{-*}V_2{|\Lambda_2|}^{-\frac{1}{2}}$
\State $R_2 \leftarrow {|\Lambda_2|}^{\frac{1}{2}} V_2^* L_2^*$
\Statex \textbf{// Compute new signature matrix:}
\State $\hat{\Sigma}\leftarrow\Lambda_2|\Lambda_2|^{-1}$
\end{algorithmic}
\end{algorithm}

If one is only interested in computing $H$ and $\hat{\Sigma}$, then Steps 4 and 8, computing $R_1$ and $R_2$, can be omitted.

\section{The QDWH algorithm for computing the standard polar decomposition}\label{Sec:QDWH}
Methods for the computation of the polar decomposition of a matrix $A=UH$ \eqref{Eq:StandPolDec} have been studied extensively in recent years. Once the orthogonal polar factor is computed, the symmetric factor can be recovered via $H=U^{*}A$. $H:=(H+H^*)/2$ can be performed to guarantee numerical symmetry.

A current state-of-the-art iterative method for computing the polar factor is the QDWH algorithm \cite{NakBG10}. It is based on the well-known Halley iteration which is a member of the Pad\'e family of iterations \cite{KenL92}. The Dynamically Weighted Halley (DWH) iteration introduces the weights $a_k$, $b_k$, $c_k\in \mathbb{R}^+$ and is given as
\begin{align}\label{Eq:DWH}
 X_{k+1}=X_k(a_kI+b_kX_k^{*}X_k)(I+c_kX_k^{*}X_k)^{-1},\quad X_0=\frac{1}{\|A\|_2}A.
\end{align}
Convergence is globally guaranteed with an asymptotic cubic rate, provided $A$ has full column rank. In order to choose the weights in an optimal fashion, Iteration \eqref{Eq:DWH} is understood as an iteration acting on the singular values of the iterate $X_k$. Let $X_k=U_S\Sigma_k{V_S}^{*}$ be the SVD of $X_k$. Then one step of Iteration \eqref{Eq:DWH} yields
\begin{align}\label{Eq:SVDIteration1}
 X_{k+1} = U_Sg_k(\Sigma_k){V_S}^{*},
\end{align}
where 
\begin{align}
 g_k(x)=x\frac{a_k+b_kx^2}{1+c_kx^2}.
\end{align}
The singular value $\sigma_{i,k+1}$ of $X_{k+1}$ is hence given by a rational function acting on the singular value $\sigma_{i,k}$ of $X_k$, 
\begin{align}\label{Eq:SVDIteration2}
\sigma_{i,k+1} = g_k(\sigma_{i,k}).
\end{align}
The singular values converge to 1 as $X_k$ approaches the polar factor. Let $\ell(=:\ell_0)$ be a lower bound to the singular values of $X_0$. Due to the initial scaling with $1/\|A\|_2$ the singular values of $X_0$ lie between 0 and 1. A successful strategy for accelerating convergence can be developed by minimizing the distance of $\ell_k$, a lower bound on the singular values of $X_k$, to 1 in each iteration. This line of thoughts leads to weights chosen as
\begin{align}\label{Eq:dweights}
 a_k=h(\ell_k),\quad b_k=(a_k-1)^2/4,\quad c_k=a_k+b_k-1, \quad \ell_{k+1}=g_k(\ell_k),
\end{align}
where
\begin{align}\label{Eq:dweights2}
 h(\ell)=\sqrt{1+d}+\frac{1}{2}\sqrt{8-4d+\frac{8(2-\ell^2)}{\ell^2\sqrt{1+d}}},\quad d=\sqrt[3]{\frac{4(1-\ell^2)}{\ell^4}}.
\end{align}
The weights in \eqref{Eq:dweights} are the solutions of an optimization problem. This is how they were introduced in \cite{NakBG10}. Another derivation considers the best rank-(3,2) rational approximation of the sign function. This leads to the same weights given in  \eqref{Eq:dweights}. The latter approach can be extended to rational approximations of higher order (Zolotarev's functions), see \cite{NakF16}.

For matrices $A$ with condition number $\kappa_2(A)<10^{16}$, convergence within 6 iterations can be guaranteed using these weights \cite{NakBG10}.
A simple rewrite of the iteration \eqref{Eq:DWH} 
\begin{eqnarray}\label{Eq:CholRewrite}
\begin{split}
 \lefteqn{X_k(a_kI+b_kX_k^{*}X_k)(I+c_kX_k^{*}X_k)^{-1}} \\
& = & \frac{b_k}{c_k}X_k + \left(a_k-\frac{b_k}{c_k}\right) X_k (I + c_kX_k^{*}X_k)^{-1}
\end{split}
\end{eqnarray}
leads to two distinct implementation variants:
$(I + c_kX_k^{*}X_k)$ is a symmetric positive definite matrix and its linear solve can be done using a Cholesky factorization.

\begin{align}\label{Eq:cholDWH}
\left\{
                \begin{array}{ll}
 Z_k=I+c_kX_k^{*}X_k,\quad W_k=\text{chol}(Z_k),\\
 X_{k+1}=\frac{b_k}{c_k}X_k+\left(a_k-\frac{b_k}{c_k}\right)X_kW_k^{-1}W_k^{-*}.
 \end{array}
              \right.
\end{align}

It can also be shown that $X_k(I + c_kX_k^{\tran}X_k)^{-1}$ is equivalently computed via a QR decomposition, which leads to the actual QR-based Dynamically Weighted Halley (QDWH) iteration
\begin{align}\label{Eq:QDWH}
 \left\{
                \begin{array}{ll}
                  \begin{bmatrix}
                   \sqrt{c_k} X_k \\I
                  \end{bmatrix} = \begin{bmatrix}Q_1\\Q_2\end{bmatrix}R,
\\
                  X_{k+1} = \frac{b_k}{c_k}X_k + \frac{1}{\sqrt{c_k}}\left(a_k-\frac{b_k}{c_k}\right)Q_1Q_2^{\tran}.
                \end{array}
              \right.
\end{align}
This variant entirely avoids inversion and is proven to be backward stable \cite{NakH12}. It has, however, a higher operation count than the Cholesky variant \eqref{Eq:cholDWH}. This is why in practice the algorithm  carries out the QR-based variant \eqref{Eq:QDWH} in the first iterations and switches to the Cholesky variant \eqref{Eq:cholDWH} as soon as a reasonably conditioned iterate $X_k$ is guaranteed. This way, numerical stability of the iteration is not compromised.

The two forms of the iteration represent the connection between the QR decomposition and the Cholesky factorization described in the previous section. They are two sides of the same coin. Either the QR decomposition of $A=\begin{bmatrix}\sqrt{c_k} X\\I\end{bmatrix}$  is computed (leading to Iteration \eqref{Eq:QDWH}), or the Cholesky factorization of $A^*A= I+c_k X^*X$ (Iteration \eqref{Eq:cholDWH}) is computed and used for a linear solve. 

\section{Generalized polar decompositions}\label{Sec:GenPolDec}

\subsection{The generalized QDWH algorithm}

Iterative methods for computing the generalized polar factor can be constructed from a connection to the matrix sign function.

\begin{theorem}[Computation of the canonical generalized polar decomposition, Theorem 5.1 in \cite{HigMMetal05}]\label{Thm:CompGenPol}
 Let $A=WS$ be a matrix with an existing canonical generalized polar decomposition with respect to the orthosymmetric pair $M,N$. Let 
 \begin{align}\label{Eq:SignIterCompGenPol}
 X_{k+1} = g(X_k)=X_kh(X_k^2)
 \end{align}
 be an iteration that converges to $\sign{X_0}$, assuming it exists. $g(\cdot)$ and $h(\cdot)$ are matrix functions. Let $g(0)=0$ and for sesquilinear forms assume that $g(X^{\star_N})=g(X)^{\star_N}$ holds for all $X$ in the domain of $g$. Then the iteration
 \begin{align*}
  Y_{k+1}=Y_kh(Y_k^{\star_{M,N}}Y_k),\qquad Y_0 = A,
 \end{align*}
 converges to $W$ with the same order of convergence as iteration \eqref{Eq:SignIterCompGenPol} converges to $\sign{X_0}$.
\end{theorem}

Iterations for the matrix sign function of the form \eqref{Eq:SignIterCompGenPol} are very common and well-studied \cite[Ch. 5]{Hig08}. They include the class of Pad\'{e} iterations devised in \cite{KenL91b}. Here, the iteration is given as a rational function of the form
\begin{align*}
 X_{k+1} = X_k p_{lm}(I-X_k^2)q_{lm}(I-X_k^2)^{-1},\qquad X_0=A,
\end{align*}
where $p_{lm}(\cdot)$ and $q_{lm}(\cdot)$ are explicitly given polynomials, yielding the Pad\'{e} approximant of degree $(l,m)$.

Choosing $l=m=1$ leads to the Halley iteration, which also forms the basis of the QDWH algorithm presented in Section \ref{Sec:QDWH}. In the context of the generalized polar decomposition, the dynamically weighted Halley iteration follows from applying Theorem \ref{Thm:CompGenPol} and is given as
\begin{align}\label{Eq:GenPolIter}
 X_{k+1}&=X_k(a_kI+b_kX_k^{\star_{M,N}}X_k)(I+c_kX_k^{\star_{M,N}}X_k)^{-1},\quad X_0=sA,
\end{align}
where $s\in\mathbb{K}$ is an arbitrary scaling factor, as any $sA$ has the same polar factor $W$. A discussion on how to choose a beneficial $s$ follows later.
More explicitly, using \eqref{Eq:StarMN}, \eqref{Eq:GenPolIter} is given as
\begin{align*}
 X_{k+1}&=X_k(a_kI+b_kN^{-1}X_k^*MX_k)(I+c_kN^{-1}X_k^*MX_k)^{-1},\quad X_0=sA.
\end{align*}

The generalization of the DWH algorithm given in the previous paragraphs is straightforward. We now investigate whether this iteration has attractive numerical properties and under which circumstances it can lead to an accelerated convergence. The key observation in the standard setting is that one iteration step acts as a rational function on the singular values of the iterate $X_k$ (see Equations \eqref{Eq:SVDIteration1} to \eqref{Eq:SVDIteration2}). 
A similar observation helps in the indefinite setting. 

\begin{corollary}\label{Co:IterationOnS}
 Let the canonical generalized polar decomposition $A=WS$ exist and be computed via an iteration $X_{k+1}=X_kh(X_k^{\star_{M,N}}X_k)$, $X_0=A$, as given in Theorem \ref{Thm:CompGenPol}. Then $X_k$ has a canonical generalized polar decomposition
 \begin{align*}
  X_k = W S_k.
 \end{align*}
 For the series of self-adjoint factors $S_k$ it holds 
 \begin{align}
  S_{k+1} = S_kh(S_k^2)·\label{Eq:SignSelfAdj}.
 \end{align}
\end{corollary}
\begin{proof}
See proof of Theorem 5.1 in \cite{HigMT10}.
\end{proof}
Using the Jordan canonical form $S=ZJZ^{-1}$, we see that \eqref{Eq:SignSelfAdj} is equivalent to\begin{align*}
 S_{k+1} &= Z g(J_k) Z ^{-1}=ZJ_{k+1}Z^{-1},
\end{align*} with $g(x)=xh(x^2)$. Essentially, one iteration step for computing the generalized polar decomposition acts as a rational function on the eigenvalues of the self-adjoint factor $S$, such that they converge towards 1 (or stay 0 in the rank-deficient case). Note that all non-zero eigenvalues of $S$ have positive real part and $S=(A^{\star_{M,N}}A)^{\frac{1}{2}}$ by definition.

In the standard setting outlined in Section \ref{Sec:QDWH}, i.e.\  the case $M=I_m$, $N=I_n$, $S$ is symmetric (respectively Hermitian) and has only real eigenvalues. These eigenvalues are the singular values of $A$. This property does not hold in the general case. Only the convergence of the real eigenvalues of $S$ is guaranteed to benefit from choosing the weighting parameters as in the standard case.

The reason we are interested in developing this method further, lies in its possible applications laid out in Section \ref{Sec:Introduction}. In the application in quantum physics, the relevant eigenvalues are in fact often real. This follows from physical constraints and does not follow directly from the given matrix structure. More specifically, it holds that $\Sigma A$ is Hermitian and positive definite. We call a matrix with this property a definite pseudosymmetric matrix. This property leads to $A$ having only real eigevalues (see e.g.\ \cite[Thm. 5]{BenP20c}), s.t.\ the pseudosymmetric polar factor has only positive real eigenvalues.  In this case, we expect great benefits from choosing the weighting parameters as in \eqref{Eq:dweights} and \eqref{Eq:dweights2}. 

The scaling factor $s$ in \eqref{Eq:GenPolIter} should be chosen in the following way. Let $sA=WS_s$ be the generalized polar decomposition of $X_0=sA$. The polar factor $W$ is the same as for $A$. The pseudosymmetric factor $S_s$ is the scaled pseudosymmetric factor of $A=WS$, $S_s=sS$. $s$ should be chosen such that its eigenvalues lie between 0 and 1, i.e.\ 
\begin{align}\label{Eq:sScal}
 s \leq (\lambda_\text{max}(S))^{-1} = (\lambda_\text{max}((\Sigma A^* \Sigma A)^{\frac{1}{2}}))^{-1}.
\end{align}
 $\ell$ should be a lower bound on the smallest eigenvalue of $S_s$, i.e.\
\begin{align}\label{Eq:lLow}
 \ell \leq\lambda_\text{min}(S_s)=s \lambda_{\text{min}}((\Sigma A^* \Sigma A)^{\frac{1}{2}}).
\end{align}
Computing values fulfilling \eqref{Eq:sScal} and \eqref{Eq:lLow} seems non-trivial, as computing $S$ (after computing $W$ via the iteration) is the goal of the algorithm and $S$ is not known a-priori.
The following lemma gives a remedy for square matrices.

\begin{lemma}\label{Lem:Bounds}
 Let $A\in\mathbb{K}^{n\times n}$ and $Q_1\in\mathbb{K}^{n\times n}$, $Q_2\in\mathbb{K}^{n\times n}$ be unitary. Then
 \begin{align*}
  |\lambda_{\text{max}}((Q_1 A^*Q_2 A)^\frac{1}{2})| \leq \sigma_\text{max}(A),\qquad
  |\lambda_{\text{min}}((Q_1 A^*Q_2 A)^\frac{1}{2})| \geq \sigma_\text{min}(A). 
 \end{align*}
\end{lemma}
\begin{proof}
Because the spectral norm is submultiplicative, we have 
\begin{align*}
|\lambda_{\text{max}}(Q_1 A^*Q_2 A)| \leq \sigma_{\text{max}}(Q_1 A^*Q_2 A) \leq \sigma_{\text{max}}(Q_1 A^*Q_2) \sigma_{\text{max}}(A)= \sigma_{\text{max}}(A)^2,\\
|\lambda_{\text{min}}(Q_1 A^*Q_2 A)| \geq \sigma_{\text{min}}(Q_1 A^*Q_2 A) \geq \sigma_{\text{min}}(Q_1 A^*Q_2) \sigma_{\text{min}}(A)= \sigma_{\text{min}}(A)^2.
\end{align*}
The proposition follows immediately.
\end{proof}
Lemma \ref{Lem:Bounds} for $Q_1=Q_2=\Sigma$ implies that $s$ and $\ell_0$ can be chosen as 
\begin{align}\label{Eq:chooseSL}
 s\approx 1/\sigma_{\text{max}}(A),\qquad \ell_0\approx s\sigma_{\text{min}}(A)=1/\text{cond}_2(A)
\end{align}
in order to fulfill \eqref{Eq:sScal} and \eqref{Eq:lLow} in the case of square matrices.

Additionally to favorable convergence properties guaranteed for certain matrices, generalizing the ideas from QDWH leads to a new class of inverse-free iterations for computing the generalized polar factor. In the case of self-adjoint matrices, this polar factor coincides with the matrix sign function, which is relevant in many application areas. Avoiding the inverse opens up the possibility of more stable methods. How exactly this is done is described in the following.

Here, the role of the orthogonal representations in QDWH is played by $(M,N)$-orthogonal matrices defined via two inner products given by two matrices $M$ and $N$. The following lemma provides a tool for substituting the inverse $(I+c_kX_k^{\star_{M,N}}X_k)^{-1}$ in Iteration \eqref{Eq:GenPolIter}.

 \begin{lemma}\label{Lem:GenQRIter}
  Let $M\in\mathbb{K}^{m\times m}$, $N\in\mathbb{K}^{n\times n}$  be nonsingular, and $M_2:=\begin{bmatrix}
                    M&\\&N
                   \end{bmatrix}$. For $X\in\mathbb{K}^{m\times n}$, $\eta\in\mathbb{K}$, let $\begin{bmatrix}\eta X\\I\end{bmatrix} =VR$ with $V= \begin{bmatrix}V_1\\V_2\end{bmatrix}\in\mathbb{K}^{(m+n)\times n}$, $R\in\mathbb{K}^{n\times n}$ nonsingular, be a decomposition.  Then 
 \begin{align*}
  \eta X(I+|\eta|^2 X^{\star_{M,N}}X)^{-1} 
  = V_1(V^{\star_{M_2,N}}V)^{-1}V_2^{\star_N}.
 \end{align*}
\end{lemma}
\begin{proof}
It holds
\begin{align*}
 &\eta X(I+|\eta|^2 X^{\star_{M,N}}X)^{-1} 
 = \eta X \left(\begin{bmatrix}\eta X\\I\end{bmatrix}^{\star_{M_2,N}}\begin{bmatrix}\eta X\\I\end{bmatrix}\right)^{-1}\\
 =&V_1 ((VR)^{\star_{M_2,N}}V)^{-1}
 = V_1(V^{\star_{M_2,N}}V)^{-1}V_2^{\star_N}.
\end{align*}
In the last step we used $V_2=R^{-1}$.
\end{proof}

For $M=I_m$, $N=I_n$ and orthogonal or unitary $V$, we have the known result
\begin{align*}
  \eta X(I+|\eta|^2X^{*}X)^{-1} = V_1V_2^*,
 \end{align*}
 given for example as Theorem 4.1 in \cite{NakBG10}.  The original QDWH algorithm is based on this result. A straightforward idea to generalize this approach would be to choose $V$ to be $(M_2,N)$-orthogonal, i.e.\  $V^{\star_{M_2,N}}V=I$. The next lemma shows how we can relax this condition, while keeping the inverse easy to compute.

\begin{lemma}\label{Lem:M2Nisometry}
 Let $M\in\mathbb{K}^{m\times m}$, $N\in\mathbb{K}^{n\times n}$ be nonsingular, $M_2=\begin{bmatrix}
                    M&\\&N
                   \end{bmatrix}$, and $V=\begin{bmatrix}V_1\\V_2\end{bmatrix}\in\mathbb{K}^{(m+n)\times n}$ be $(M_2,\hat{N})$-orthogonal for a matrix $\hat{N}\in\mathbb{K}^{n\times n}$, i.e.\  $V^*M_2V=\hat{N}$. Then
 \begin{align*}
  V_1(V^{\star_{M_2,N}}V)^{-1}V_2^{\star_N} =  V_1V_2^{\star_{N,\hat{N}}}.
 \end{align*}
\end{lemma}
\begin{proof}
 From $V^*M_2V=\hat{N}$, it follows $V^{\star_{M_2,N}}V=N^{-1}\hat{N}$ and therefore
 \begin{align*}
  V_1(V^{\star_{M_2,N}}V)^{-1}V_2^{\star_N} = V_1\hat{N}^{-1}V_2^*N = V_1 V_2^{\star_{N,\hat{N}}}.
 \end{align*}

\end{proof}

\subsection{Realizing the $\Sigma$DWH iteration}\label{Sec:DWHAlgorithms}
When a practical method for computing the $(M_2,\hat{N})$-orthogonal matrices in Lemma \ref{Lem:M2Nisometry} is available, we can formulate a generalized QDWH algorithm. If $N^{-1}$ is trivial to compute, this leads to an inverse-free computation, if the computation of the $(M_2,\hat{N})$-orthogonal matrix avoids inversion. We now leave the general framework and restrict ourselves to inner products induced by signature matrices.

Section \ref{Sec:LDLHQR} laid the groundwork for several options in the algorithm design realizing the iteration for the canonical generalized polar decomposition of $A\in\mathbb{K}^{m\times n}$ \eqref{Eq:GenPolIter} with respect to the signature matrices $\Sigma_m$ and $\Sigma_n$. As signature matrices are involutory, the iteration is given as
\begin{align}\label{Eq:SigmaPolIter}
 X_{k+1}&=X_k(a_kI+b_k\Sigma_n X_k^*\Sigma_m X_k)(I+c_k\Sigma_n X_k^*\Sigma_m X_k)^{-1},\quad X_0=sA.
\end{align}
We call \eqref{Eq:SigmaPolIter} the $\Sigma$DWH iteration. The naive approach is to implement the iteration straightforward, using a linear solve employing the MATLAB backslash operator.  However, there is a better way to exploit the structure at hand. To see this, we rewrite \eqref{Eq:SigmaPolIter}
\begin{eqnarray*}
\begin{split}
 \lefteqn{X_k(a_kI+b_k\Sigma_n X_k^*\Sigma_m X_k)(I+c_k\Sigma_n X_k^* \Sigma_m X_k)^{-1}}\\
 &=&\frac{b_k}{c_k}X_k + (a_k-\frac{b_k}{c_k})X_k{(\Sigma_n +c_k X_k^*\Sigma_m X_k)}^{-1}\Sigma_n.
\end{split}
\end{eqnarray*}
This is the indefinite analogue to \eqref{Eq:CholRewrite}. In the standard case, the Cholesky factorization is employed to exploit the symmetric structure in the iteration \eqref{Eq:cholDWH}. In the indefinite case, this role is played by the pivoted $LDL^{\tran}$ factorization.
Analogous to \eqref{Eq:cholDWH},  Iteration \eqref{Eq:SigmaPolIter} is equivalently given as
\begin{align}\label{Eq:LDLDWH}
\left\{
                \begin{array}{ll}
 Z_k=\Sigma_n+c_k X_k^* \Sigma_m X_k,\quad [L_k,D_k,P_k]=\text{ldl}(Z_k),\\
 X_{k+1}=\frac{b_k}{c_k}X_k+\left(a_k-\frac{b_k}{c_k}\right)X_kP_kL_k^{-*}D_k^{-1}L_k^{-1}P_k^{\tran}\Sigma.
 \end{array}
              \right.
\end{align}
This approach is already more promising than the naive one because the structure of the involved matrices is exploited. This way, less computational work is needed and we may expect better accuracy.
We employ Lemma \ref{Lem:GenQRIter} and Lemma \ref{Lem:M2Nisometry} to find an equivalent formulation of the DWH iteration \eqref{Eq:SigmaPolIter}, which in principle does not rely on computing inverses. The role of $\hat{N}$ in Lemma \ref{Lem:M2Nisometry} is played by another signature matrix $\hat{\Sigma}_n$ of size $n\times n$.  
The formulation
\begin{align}\label{Eq:HQDWH}
 \left\{
                \begin{array}{ll}
                  \begin{bmatrix}
                   \sqrt{c_k} X_k \\I
                  \end{bmatrix} = \begin{bmatrix}H_1\\H_2\end{bmatrix}R,\text{ where } \begin{bmatrix}H_1\\H_2\end{bmatrix}^*\begin{bmatrix}\Sigma_m&\\&\Sigma_n\end{bmatrix}\begin{bmatrix}H_1\\H_2\end{bmatrix} = \hat{\Sigma}_n,
\\
                  X_{k+1} = \frac{b_k}{c_k}X_k + \frac{1}{\sqrt{c_k}}\left(a_k-\frac{b_k}{c_k}\right)H_1\hat{\Sigma}_n H_2^{*}\Sigma_n
                \end{array}
              \right.
\end{align}
is the analog to the QR-based iteration \eqref{Eq:QDWH} in the standard case. 
Instead of an orthogonal basis (using the QR decomposition), a $\left(\begin{bmatrix}\Sigma_m&\\&\Sigma_n\end{bmatrix},\hat{\Sigma}_n\right)$-orthogonal basis is computed. This can be done by computing the hyperbolic QR decomposition (Theorem \ref{Thm:HypQR}) or the indefinite QR decomposition (Theorem \ref{Thm:IQR}). Here, methods exist that are based on successive column elimination and do not perform any matrix inversions. Computing the indefinite QR decomposition via an $LDL^{\tran}$ factorization (i.e.\  employing Lemma \ref{Thm:HermIndFac1}) gives exactly the $LDL^{\tran}$ based iteration \eqref{Eq:LDLDWH}.  Another promising way to compute the required basis is to employ the presented LDLIQR2 algorithm (Algorithm \ref{Alg:LDLIQR2}).

The resulting stability for an iteration employing these different approaches is examined experimentally in the numerical experiments of Section \ref{Sec:NumRes}.
 
\section{Subspaces in the $\Sigma$DWH iteration}\label{Sec:PermGraph}
\subsection{Permuted graph bases for general matrices}
Looking at Lemma \ref{Lem:GenQRIter}, we see that the factor $R$ of the $VR$ decomposition is in fact not referenced in order to rewrite part of the $DWH$ iteration. This suggests the idea to employ a well-conditioned basis of the subspace spanned by $\begin{bmatrix}                                                                                                                                                                                                                                                         \sqrt{c_k}X\\                                                                                                                                                                                                                                                         I_n                                                                                                                                                                                                                                                        \end{bmatrix}$.  The linear solve in one iteration step is not avoided completely but we hope to invert a better-conditioned matrix. 

In the following we use $A \sim B$ to indicate that the columns of the two matrices $A$ and $B$ span the same subspace.
A good candidate for providing a basis with desirable properties are permuted graph bases. An $n$-dimensional subspace $\mathcal{U}$ is said to be represented in a \emph{permuted graph basis} if
\begin{align}\label{Eq:permGraph}
 \mathcal{U} = \colspan{P^{\tran}\begin{bmatrix}
                               I_n\\X
                              \end{bmatrix}},
\end{align}
where $P$ denotes a permutation and $I_n$ is the identity matrix. It is shown in \cite{MehP12} that a permutation $P$ exists, such that the entries of $X$ are all smaller than $1$. This leads to much better numerical properties when using this representation in numerical algorithms. 

The actual computation of the entry-bound representations \eqref{Eq:permGraph} is an NP-hard problem. However, \cite{MehP12} presents heuristic methods that compute representations, where for a given threshold value $\tau >1$, $|x_{i,j}|<\tau$. This can be done with a reasonable amount of computational effort. In the worst case this is $\cO(n^3\log{n})$. In practice, it is typically much lower, in particular when good starting guesses for $P$ are available.

The following lemma is a reformulation of Lemma \ref{Lem:GenQRIter}, where $M=\Sigma_m$ and $N=\Sigma_n$ are signature matrices and $V$ is attained via representation \eqref{Eq:permGraph}.

\begin{lemma}
  Let $\Sigma_m\in\mathbb{R}^{m\times m}$ ,$\Sigma_n\in\mathbb{R}^{n\times n}$ be signature matrices. 
                   For $X\in\mathbb{K}^{n\times n}$, $\eta\in\mathbb{K}$ let $\begin{bmatrix}I\\\eta X\end{bmatrix} \sim  V =\begin{bmatrix}V_1\\V_2\end{bmatrix} = P^{\tran}\begin{bmatrix}
                               I\\\hat{X}
                              \end{bmatrix} \in\mathbb{K}^{2n\times n}$, where $P$ is a permutation. Let 
\begin{align*}
 P\begin{bmatrix}
 \Sigma_n&\\&\Sigma_m
 \end{bmatrix}P^{\tran} = \begin{bmatrix} \hat{\Sigma}_n&\\&\hat{\Sigma}_m\end{bmatrix}.
\end{align*}
 Then
 \begin{align*}
  \eta X (I+|\eta|^2 \Sigma_n X^* \Sigma_m X)^{-1}
  = V_2(\hat{\Sigma}_n + \hat{X}^*\hat{\Sigma}_m \hat{X})^{-1}V_1^* \Sigma_n.
 \end{align*}
\end{lemma}
\begin{proof}
Let $\Sigma_2:=\begin{bmatrix}\Sigma_n&\\&\Sigma_m\end{bmatrix}$. We follow the lines of the proof of Lemma \ref{Lem:GenQRIter}
As $\begin{bmatrix}I\\\eta X\end{bmatrix}$ and $V$ span the same subspace, there exists a nonsingular matrix $R$ s.t. 
\begin{align*}
\begin{bmatrix}I\\\eta X\end{bmatrix}=VR.
\end{align*}

Exactly as in the proof of Lemma \ref{Lem:GenQRIter} (with the roles of $V_1$ and $V_2$ switched) it can be shown that
\begin{align*}
 \eta X (I+|\eta|^2 \Sigma_n X^* \Sigma_m X)^{-1} &= V_2(V^{\star_{\Sigma_2,\Sigma_n}}V)^{-1}V_1^{\star_{\Sigma_n}}\\
 &=V_2(\hat{\Sigma}_n + \hat{X}^*\hat{\Sigma}_m \hat{X})^{-1}V_1^* \Sigma_n.
\end{align*}
\end{proof}
Algorithm \ref{Alg:PermGraphGenPolDec} presents the details on how permuted graph bases can be used in the computation of generalized polar decomposition via the dynamically weighted Halley iteration. 

\begin{algorithm}[ht]
\caption{Compute the generalized polar decomposition with respect to signature matrices, using permuted graph bases.} \label{Alg:PermGraphGenPolDec}
\begin{algorithmic}[1]
\Require $A\in\mathbb{K}^{m\times n}$,
\newline
$\Sigma_m\in\mathbb{R}^{m\times m}$ 
$\Sigma_n\in\mathbb{R}^{n\times n}$: signature matrices, s.t.\ the canonical generalized polar decomposition of $A$ exists (according to Theorem \ref{Thm:CompGenPol}),\newline
 $s$: estimate on $|\lambda_{\text{max}}((\Sigma_n A^*\Sigma_m A)^\frac{1}{2})|^{-1}$,\newline
$\ell$: estimate on $s|\lambda_{\text{min}}(\Sigma_n A^*\Sigma_m A)^\frac{1}{2}|$,\newline $\tau>1$: threshold value for permuted graph basis.
\Ensure $A=WS$ is the canonical generalized polar decomposition with respect to $\Sigma_m$ and $\Sigma_n$.
\State $U\leftarrow sA$.
\For{$k=1,2,\dots$}
\State Compute weighting parameters $a$, $b$, $c$ and update $\ell$ from equations \eqref{Eq:dweights} and \eqref{Eq:dweights2}.
\State Compute entry-bound permuted graph bases of $\colspan{\begin{bmatrix}I\\\sqrt{c}W\end{bmatrix}}$, i.e.
\begin{gather*}
 \begin{bmatrix}I\\\sqrt{c}W\end{bmatrix} \sim P^{\tran}\begin{bmatrix}I\\\hat{W}\end{bmatrix}=:
        \begin{bmatrix}
         V_1\\V_2
        \end{bmatrix},\\ |\hat{W}_{ij}| < \tau \text{ for } i\in\{1,\dots,m\},\ j\in\{1,\dots,n\}.      
\end{gather*}
\State 
$
        \begin{bmatrix}
         \hat{\Sigma}_n&\\
         &\hat{\Sigma}_m
        \end{bmatrix}\leftarrow P\begin{bmatrix}
         \Sigma_n&\\
         &\Sigma_m
        \end{bmatrix}P^{\tran}
        $
\State Compute $LDL^{\tran}$ factorization $\hat{\Sigma}_n +\hat{W}^* \hat{\Sigma}_m \hat{W} = PLDL^*P^{\tran}$.
\State\label{Alg:PermGraphGenPolDec:Update}$W \leftarrow  \frac{b}{c} W + (a-\frac{b}{c})V_2PL^{-*}D^{-1}L^{-1}P^{\tran}V_1^{*}\Sigma_n$
\EndFor
\State Compute pseudosymmetric factor and ensure pseudosymmetry numerically
\Statex $S\leftarrow \Sigma_n W^* \Sigma_m A, \qquad S\leftarrow (S+\Sigma_n S^* \Sigma_n)/2.$
\end{algorithmic}
\end{algorithm}

\subsection{Permuted Lagrangian graph bases for pseudosymmetric matrices}
As pointed out in Section \ref{Sec:Introduction}, we are in particular interested in computing the generalized polar decomposition (with respect to a signature matrix) of pseudosymmetric matrices. A way to exploit this structure in the iteration can be found by considering Lagrangian subspaces, to which pseudosymmetric matrices can be linked.

A subspace $\mathcal{U}=\colspan{U}$, $U\in\mathbb{K}^{2n\times n}$, is called \emph{Lagrangian} if it holds $U^*JU = 0$,
where $J=\begin{bmatrix}
          0&I_n\\-I_n&0
         \end{bmatrix}$.

A Lagrangian subspace can be represented by a \emph{permuted Lagrangian graph basis}
\begin{align}\label{Eq:permLangGraph}
\mathcal{U} = \colspan{\Pi^{\tran}\begin{bmatrix}
                               I\\X
                              \end{bmatrix}},
\end{align}
where $X=X^*$. $\Pi$ denotes a symplectic swap matrix  \cite{Ben01}. A symplectic swap matrix is defined by a boolean vector $v\in\{0,1\}^n$ and its complement $\hat{v}\in\{0,1\}^n$, $\hat{v}_i = 1- v_i$. The corresponding symplectic swap matrix is defined as
\begin{align}\label{Eq:SympSwap}
 \Pi_v = \begin{bmatrix}
          \diag{v} & \diag{\hat{v}}\\
          -\diag{\hat{v}} & \diag{v}
         \end{bmatrix}.
\end{align}

It is shown in \cite{MehP12} that each Lagrangian subspace admits a representation \eqref{Eq:permLangGraph}, where $X$ has no entries with modulus larger than $\sqrt{2}$.  

As for general subspaces, there exist heuristics for computing a basis, such that the entries of $X$ are bounded, within a reasonable amount of time. In this case $|x_{i,j}|<\tau$, where $\tau>\sqrt{2}$ is a given threshold value.

A Lagrangian subspace could of course be treated as a general subspace and admits a representation \eqref{Eq:permGraph}, with even smaller entries than in \eqref{Eq:permLangGraph}. However, the structural property, i.e.\ the subspace being Lagrangian, is not encoded anymore in this representation. It is encoded in the symmetry of $X$, which can easily be enforced and preserved in the course of computations. This has numerical benefits, which typically outweigh the slightly larger entries in $X$.

The following lemma draws a connection between self-adjoint matrices and Lagrangian subspaces.

\begin{lemma}\label{Lem:SelfAdjLagr}
 Let $M\in\mathbb{K}^{n\times n}$, $M=M^*$ be a nonsingular matrix. Let $X\in\mathbb{K}^{n\times n}$ be self-adjoint with respect to the inner product induced by $M$. Then $\begin{bmatrix}
 M\\
 X
 \end{bmatrix}$ spans a Lagrangian subspace.
\end{lemma}
 The following lemma is a variant of Lemma \ref{Lem:GenQRIter} applied to square matrices, where the positions of the two matrix blocks are switched. The goal is to get to a formulation, in which the subspace given in Lemma \ref{Lem:SelfAdjLagr} appears.
 
 \begin{lemma}\label{Lem:GenQRIter2}
  Let $M,N\in\mathbb{K}^{n\times n}$ be nonsingular, $N$ be $M$-orthogonal, i.e.\  $N^{\star_M}N=I$. $M_2:=\begin{bmatrix}
                    M&\\&M
                   \end{bmatrix}$,
  $X\in\mathbb{K}^{n\times n}$. Let $\begin{bmatrix}N\\ \eta X\end{bmatrix} =VR$ with $V= \begin{bmatrix}V_1\\V_2\end{bmatrix}\in\mathbb{K}^{2n\times n}$, $R\in\mathbb{K}^{n\times n}$ nonsingular be a decomposition.  Then 
 \begin{align*}
  \eta X(I+|\eta|^2 X^{\star_{M}}X)^{-1} 
  = V_2(V^{\star_{M_2,M}}V)^{-1}V_1^{\star_M} N.
 \end{align*}
 \end{lemma}
\begin{proof}
 We observe
 \begin{align*}
  \begin{bmatrix}N\\ \eta X\end{bmatrix}^{\star_{M_2,M}}\begin{bmatrix}N\\ \eta X\end{bmatrix} = N^{\star_M}N + |\eta|^2 X^{\star_M}X  = I + |\eta|^2X^{\star_M}X.
 \end{align*}
Following the proof of Lemma \ref{Lem:GenQRIter}, we get
\begin{align*}
 \eta X(I+|\eta|^2 X^{\star_{M}}X)^{-1}  = V_2(V^{\star_{M_2,M}}V)^{-1}(R^{-1})^{\star_M} = V_2(V^{\star_{M_2,M}}V)^{-1}V_1^{\star_M} N.
\end{align*}
In the last step we used $R^{-1}=N^{-1}V_1= N^{\star_M}V_1 $.
\end{proof}

Let us go back to the specific case of an inner product induced by a signature matrix, i.e.\  $M:=\Sigma$. In this case, Lemma \ref{Lem:SelfAdjLagr} and Lemma \ref{Lem:GenQRIter2} come together. $\Sigma$ is symmetric, so Lemma \ref{Lem:SelfAdjLagr} holds. So does Lemma \ref{Lem:GenQRIter2} by setting $N:=\Sigma$. The subspace in question can be represented by permuted Lagrangian graph bases. The situation is summarized in the following lemma. 
 
 \begin{lemma}\label{Lem:LagrIter}
  Let $\Sigma\in\mathbb{R}^{n\times n}$ be a signature matrix. $\Sigma_2:=\begin{bmatrix}
                    \Sigma&\\&\Sigma
                   \end{bmatrix}$,
  $X\in\mathbb{K}^{n\times n}$ be self-adjoint with respect to the inner product induced by $\Sigma$, $\eta\in\mathbb{K}$. Let 
  \begin{align*}
\begin{bmatrix}\Sigma\\ \eta X\end{bmatrix} \sim \Pi^{\tran}
  \begin{bmatrix}
   I\\\hat{X}
   \end{bmatrix} 
   =:V=
   \begin{bmatrix}
    V_1\\V_2
   \end{bmatrix}
  \end{align*}
 be a permuted Lagrangian graph basis, i.e.\  $\Pi$ is a symplectic swap matrix and $\hat{X}=\hat{X}^*$.  Then 
 \begin{align*}
  \eta X(I+|\eta|^2 \Sigma X^*\Sigma X)^{-1} 
  = V_2(\Sigma + \hat{X} \Sigma \hat{X})^{-1} V_1^*.  
 \end{align*}
 \end{lemma}
\begin{proof}
Note that
\begin{align*}
 V^{\star_{\Sigma_2,\Sigma}}V = 
 \Sigma \begin{bmatrix}
  I_{2n} & \hat{X}^{\tran}
 \end{bmatrix}\Pi \Sigma_2 \Pi^{\tran} 
 \begin{bmatrix}
  I_{2n}\\
  \hat{X}
 \end{bmatrix} = I_{2n} + \Sigma \hat{X}^* \Sigma \hat{X}= I_{2n} + \Sigma \hat{X} \Sigma \hat{X}.
\end{align*}
We have used $\Pi \Sigma_2 \Pi^{\tran}=\Sigma_2$, which holds because $\Pi=\begin{bmatrix}V & \hat{V}\\
   -\hat{V} & V                                                                                                                                                                                                                                                                                                              \end{bmatrix}$ is a symplectic swap matrix as given in \eqref{Eq:SympSwap}: 
 \begin{eqnarray*}
  \Pi \Sigma_2 \Pi^{\tran{}} = \begin{bmatrix}V & \hat{V}\\
   -\hat{V} & V                                                                                                                                                                                                                                                                                                              \end{bmatrix} \begin{bmatrix}
                    \Sigma&\\&\Sigma
                   \end{bmatrix}\begin{bmatrix}V & \hat{V}\\
   -\hat{V} & V                                                                                                                                                                                                                                                                                                              \end{bmatrix}^{\tran} = 
   \begin{bmatrix}
    V\Sigma V + \hat{V}\Sigma\hat{V} & -V\Sigma\hat{V} + \hat{V}\Sigma V\\
    \hat{V}\Sigma V -V\Sigma\hat{V} & \hat{V}\Sigma \hat{V} + V\Sigma V
   \end{bmatrix}   
   =\Sigma_2.
 \end{eqnarray*}$V\Sigma V + \hat{V}\Sigma\hat{V}=\Sigma$ and $-V\Sigma\hat{V} + \hat{V}\Sigma V=0$  hold because $V$ and $\hat{V}$ pick up complementing rows and columns of $\Sigma$. Now applying Lemma \ref{Lem:GenQRIter2} gives
\begin{align*}
\eta X(I+|\eta|^2 X^{\star_{\Sigma}}X)^{-1} 
  = V_2(\Sigma + \hat{X}^* \Sigma \hat{X})^{-1} V_1^*.  
\end{align*}
\end{proof}

Algorithm \ref{Alg:PermLangGraphGenPolDec} is a variant of Algorithm \ref{Alg:PermGraphGenPolDec} using permuted Lagrangian graph bases. It computes the generalized polar decomposition of a pseudosymmetric matrix with respect to its defining signature matrix. 
\begin{algorithm}[ht]
\caption{Compute the generalized polar decomposition of a pseudosymmetric matrix with respect to a signature matrix, using permuted Lagrangian graph bases.} \label{Alg:PermLangGraphGenPolDec}
\begin{algorithmic}[1]
\Require Signature matrix $\Sigma\in\mathbb{K}^{n\times n}$,  $A= \Sigma A^* \Sigma \in\mathbb{K}^{n\times n}$,  s.t.\ $A$ has no purely imaginary eigenvalues, \newline
$s$: estimate on $|\lambda_{\text{max}}((\Sigma A^*\Sigma A)^\frac{1}{2})|^{-1}$,\newline
$\ell$: estimate on the norm of the smallest eigenvalue of $s(\Sigma A^*\Sigma A)^\frac{1}{2}$,\newline
$\tau>\sqrt{2}$: threshold value for permuted Lagrangian graph bases.

\Ensure $A=WS$ is the generalized polar decomposition with respect to $\Sigma$.
\State $W\leftarrow A/\|A\|_2$.
\For{$k=1,2,\dots$}
\State Compute weighting parameters $a$, $b$, $c$ and update $\ell$  from equations \eqref{Eq:dweights} and \eqref{Eq:dweights2}.
\State Compute entry-bound permuted Lagrangian graph bases of $\colspan{\begin{bmatrix}\Sigma\\\sqrt{c}W\end{bmatrix}}$, i.e.\ 
\begin{align*}
 \begin{bmatrix}\Sigma \\\sqrt{c}W\end{bmatrix} \sim \Pi^{\tran}\begin{bmatrix}I\\\hat{W}\end{bmatrix}=:
        \begin{bmatrix}
         V_1\\V_2
        \end{bmatrix},\quad |\hat{W}_{ij}| < \tau \text{ for } i,j\in\{1,\dots,n\}.
\end{align*}
\State Compute $LDL^{\tran}$ factorization $\Sigma +\hat{W}^* \Sigma \hat{W} = PLDL^*P^{\tran}$.
\State \label{Alg:PermLangGraphGenPolDec:Update} $W \leftarrow  \frac{b}{c} W + (a-\frac{b}{c})V_2PL^{-*}D^{-1}L^{-1}P^{\tran}V_1^*$
\EndFor
\State Compute pseudosymmetric factor and ensure pseudosymmetry numerically
\Statex $S\leftarrow \Sigma W^* \Sigma A, \qquad S\leftarrow (S+\Sigma S^* \Sigma)/2.$
\end{algorithmic}
\end{algorithm}

In the update step (Step \ref{Alg:PermGraphGenPolDec:Update} in Algorithm \ref{Alg:PermGraphGenPolDec} and Step \ref{Alg:PermLangGraphGenPolDec:Update} in Algorithm \ref{Alg:PermLangGraphGenPolDec}), the structure of $V_1$ and $V_2$ should be taken into account for an efficient implementation. 
The rows of the identity matrix are distributed in $V_1$ and $V_2$ according to the permutation $P$ or the symplectic swap $\Pi$. The remaining columns are given by $\hat{W}$. If this is taken care of, the matrix representing the subspace $V=\Pi^ {\tran}\begin{bmatrix}I\\\hat{U}\end{bmatrix}$ never has to actually be formed. We can directly work on the matrices $W$ and $\hat{W}$. 

However, we may need to form a $n\times 2n$ matrix if a good starting guess for the permutation in the computation of the permuted graph basis is desired. For this task,   a heuristic is proposed in \cite{MehP12} that includes a modified version of the QR factorization with column pivoting of an $n\times 2n$ matrix.

\section{Numerical results}\label{Sec:NumRes}
In this paper, we have developed several variants of the $\Sigma$DWH iteration to compute the canonical generalized polar decomposition of a matrix with respect to signature matrices. 

In general, the existence of the decomposition is not guaranteed, which is why we first examine pseudosymmetric matrices with respect to $\Sigma$. For these matrices, the generalized polar decomposition exists if and only if $A$ has no purely imaginary eigenvalues (note that this is also required for $\sign{A}$ to exist). For randomly generated matrices this is typically the case, which is why we observe convergence most times. Pseudosymmetric matrices represent an important class of matrices regarding the application potential of the developed methods, as pointed out in Section \ref{Sec:Introduction}. For other matrices, which are not pseudosymmetric but yield a generalized polar decomposition with respect to $\Sigma$, similar results were observed in further tests. All experiments were performed in MATLAB R2017a.

In light of the asymptotic cubic convergence of the iteration (see \cite[Sec. 4.9.2]{Hig08}) we use the stopping criterion
\begin{align}\label{Eq:StopCrit}
 \| X_{k+1} - X_{k} \|_F \leq (5\epsilon)^{\frac{1}{3}},
\end{align}
 where $\epsilon$ is the machine precision.  

We take the same values for $s$ and $\ell$ as in the QDWH algorithm \cite{NakBG10}, which are given in \eqref{Eq:chooseSL}. As explained there, this makes sense for definite pseudosymmmetric matrices. The resulting convergence behavior is the same as in the standard setting. 
Further investigation of the convergence behavior is needed to devise sensible values for $s$ and $\ell$ in the general case. Here the iteration may act on complex values. This consideration goes beyond the scope of this paper. We use the same values as in the definite case even when they are not completely justified.

We first compare the algorithms in terms of their achieved residual for badly conditioned matrices. We consider square matrices and their generalized polar decomposition for a given signature matrix $\Sigma:=\begin{bmatrix}I_n&\\&-I_n\end{bmatrix}$ ($M=N=\Sigma$ in Definition \ref{Def:GenPolDec}).

\paragraph{Example 1}
A real pseudosymmetric matrix with a condition number $\kappa = 10^k$ is generated as $A:=\Sigma Q D Q^ {\tran}$. $Q$ is a random orthogonal matrix (\texttt{orth(rand(2*n))}), and $D$ is a diagonal matrix containing equally distributed values between $1$ and $10^k$, with alternating signs. A polar decomposition $A\approx WS$ is computed and the resulting residual $\|WS-A\|_{F}/\|A\|_F$ for matrices of size $200\times 200$ ($n=100$) is given in Figure \ref{Fig:SigmaQDWH_Stability}. The residuals were averaged over 10 runs with different randomly generated matrices.\par

\begin{figure}
 \newlength\figureheight
\newlength\figurewidth
\setlength\figureheight{0.4\textwidth}
\setlength\figurewidth{0.8\textwidth}
%
%
\definecolor{mycolor1}{rgb}{0.00000,0.44700,0.74100}%
\definecolor{mycolor2}{rgb}{0.85000,0.32500,0.09800}%
\definecolor{mycolor3}{rgb}{0.92900,0.69400,0.12500}%
\definecolor{mycolor4}{rgb}{0.49400,0.18400,0.55600}%
\definecolor{mycolor5}{rgb}{0.46600,0.67400,0.18800}%
\begin{tikzpicture}

\begin{axis}[%
width=0.951\figurewidth,
height=\figureheight,
at={(0\figurewidth,0\figureheight)},
scale only axis,
xmode=log,
xmin=1,
xmax=1e15,
xminorticks=true,
xlabel style={font=\color{white!15!black}},
xlabel={cond(A)},
ymode=log,
ymin=1e-15,
ymax=1.5,
yminorticks=true,
ylabel style={font=\color{white!15!black}},
ylabel={Residual $\|WS-A\|_{\text{fro}}/\|A\|_{\text{fro}}$},
axis background/.style={fill=white},
axis x line*=bottom,
axis y line*=left,
legend style={at={(0.05,0.5)}, anchor=south west, legend cell align=left, align=left, draw=none, fill=none}
]
\addplot [color=mycolor1, mark=o, mark options={solid, mycolor1}]
  table[row sep=crcr]{%
10	3.47997703442597e-15\\
100	7.12666538738549e-14\\
1000	2.70604152854802e-13\\
10000	1.25641888820759e-11\\
100000	5.55953624107927e-10\\
1000000	1.89964687476253e-08\\
10000000	7.30473353221498e-07\\
100000000	1.83074324402766e-05\\
1000000000	0.000365172506640528\\
10000000000	0.0258796395265902\\
100000000000	0.0259483030171659\\
1000000000000	0.0387020903967116\\
10000000000000	0.035529083361234\\
100000000000000	0.0584105528197845\\
1e+15	0.0571939714389267\\
};
\addlegendentry{Backslash}

\addplot [color=mycolor2, mark=+, mark options={solid, mycolor2},densely dotted]
  table[row sep=crcr]{%
10	8.20509737453684e-16\\
100	1.39904032234934e-14\\
1000	4.28990923696599e-15\\
10000	3.7062101774673e-14\\
100000	3.68026563455984e-13\\
1000000	3.3804761480605e-12\\
10000000	1.9624161592583e-11\\
100000000	2.56943372866637e-10\\
1000000000	1.15707006021766e-09\\
10000000000	7.71258491841916e-09\\
100000000000	3.43323172486706e-07\\
1000000000000	2.79198640741701e-09\\
10000000000000	2.04880929342889e-09\\
100000000000000	2.20723307204552e-10\\
1e+15	3.31549190139513e-11\\
};
\addlegendentry{LDL}

\addplot [color=mycolor3, mark=asterisk, mark options={solid, mycolor3}, dashed]
  table[row sep=crcr]{%
10	2.6081698490087e-12\\
100	1.63988288162473e-10\\
1000	1.35455122630068e-11\\
10000	3.94483181526401e-11\\
100000	5.07030651486565e-11\\
1000000	3.03061175579705e-11\\
10000000	3.46188675364959e-11\\
100000000	2.78762900817632e-11\\
1000000000	5.47965835450489e-11\\
10000000000	1.31427036573332e-10\\
100000000000	3.54281886337334e-10\\
1000000000000	6.29026201995437e-10\\
10000000000000	6.60818120140768e-09\\
100000000000000	4.96811599875952e-09\\
1e+15	4.82113258308929e-07\\
};
\addlegendentry{Hyperbolic QR}

\addplot [color=mycolor4, mark=star, mark options={solid, mycolor4},densely dashed]
  table[row sep=crcr]{%
10	9.97107425437097e-15\\
100	4.00848535387901e-14\\
1000	6.28055207797586e-14\\
10000	6.80029670260703e-14\\
100000	4.89747902816565e-14\\
1000000	9.18069889100666e-14\\
10000000	3.21194892314017e-14\\
100000000	5.22812663694596e-14\\
1000000000	7.20690583996626e-14\\
10000000000	4.49960924599387e-14\\
100000000000	7.33799287800149e-14\\
1000000000000	1.73962397203231e-13\\
10000000000000	6.18174095620987e-14\\
100000000000000	2.69118874065461e-14\\
1e+15	5.50776754788221e-14\\
};
\addlegendentry{LDLIQR2}

\addplot [color=mycolor5, mark=square, mark options={solid, mycolor5},densely dashdotdotted]
  table[row sep=crcr]{%
10	6.9207046686701e-16\\
100	2.83272519495585e-15\\
1000	1.75868239337135e-15\\
10000	5.35846356307863e-15\\
100000	2.47124580858594e-15\\
1000000	2.51581539164198e-15\\
10000000	2.79995515174389e-15\\
100000000	3.66917118806044e-15\\
1000000000	3.84391760430639e-15\\
10000000000	2.16970552625374e-15\\
100000000000	2.17243208753813e-15\\
1000000000000	1.72904072300165e-15\\
10000000000000	1.78992275162387e-15\\
100000000000000	3.02308047016748e-15\\
1e+15	3.45430089748014e-15\\
};
\addlegendentry{PLG}

\end{axis}
\end{tikzpicture}%
 \caption{ Residuals for different iterations for computing the generalized polar decomposition of pseudosymmetric matrices $A\in\mathbb{R}^{200\times 200}$ with a certain condition number. ``Backslash'' refers to the naive implementation, ``LDL'' refers to iteration \eqref{Eq:LDLDWH}, ``Hyperbolic QR'' and ``LDLIQR2'' refer to the variants of iteration \ref{Eq:HQDWH}. ``PLG'' refers to the variant using permuted Lagrangian graph bases described in Algorithm \ref{Alg:PermLangGraphGenPolDec}. \label{Fig:SigmaQDWH_Stability}}
\end{figure}
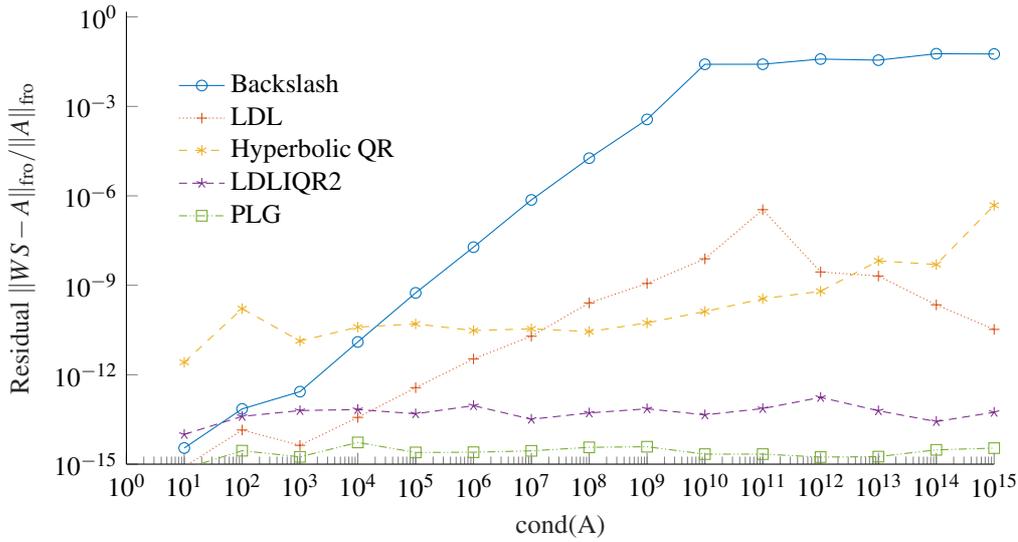 

We see that a naive implementation of the $\Sigma DWH$ iteration \eqref{Eq:SigmaPolIter} leads to a highly unstable method. The accuracy improves as the iteration is rewritten to employ the $LDL^{\tran}$ decomposition (see \eqref{Eq:LDLDWH}). This can be interpreted as exploiting structure that is hidden and ignored in the original formulation. Again the accuracy deteriorates as the matrix becomes ill-conditioned. Surprisingly, for matrices with a condition number higher than $10^{11}$, this trend is reversed and the method performs quite well for extremely ill-conditioned matrices. A possible explanation is that MATLAB function \texttt{ldl} estimates the condition number of the input and acts differently, in our case preferably, for ill-conditioned matrices.  
The $LDL^{\tran}$-based iteration can be read as an iteration based on the indefinite QR decomposition (see Theorem \ref{Thm:IQR} and iteration \eqref{Eq:HQDWH}), that has been computed via the pivoted $LDL^ {\tran}$ decomposition. For computing a hyperbolic QR decomposition directly, using a column elimination approach, we used available MATLAB code \cite{Hou15}, based on the works \cite{And00,ChaS96, HenH03}. In our setting, this does not perform well. For well-conditioned matrices, this approach delivers the worst accuracy. For ill-conditioned matrices it yields better results than the naive implementation, but is still highly dependent on the condition number. The two remaining methods use the indefinite QR decomposition via a double $LDL^{\tran}$ decomposition (LDLIQR2) and permuted Lagrangian graph bases (PLG). These give high accuracy, which is independent of the condition number. For well-conditioned matrices, $LDLIQR2$ does not seem to be preferable, as it yields a higher residual than even the naive implementation. However, the residual stays at a consistently low order of magnitude as the condition number increases. Using PLGs consistently delivers the best results regarding accuracy, in the well-conditioned as well as in the ill-conditioned setting.  

The disadvantage of the PLG approach is that it relies on very recently developed, fine-grained algorithms. Therefore, no optimized implementations are available yet and the runtimes resulting from a prototype MATLAB implementation are very high. Formulating the computation of PLGs in a way that exploits current computer architectures is a challenge not yet addressed. This method would need to be block-based in order to exploit the memory hierarchy, be parallelizable and avoid communication. The LDLIQR2 approach on the other side is easily implemented and only relies on the $LDL^ {\tran}$ factorization for which highly optimized implementations are available. However, both approaches rely on pivoting strategies, implying a considerable cost for communication if they are to be deployed in a massively parallel setup.

In a practical implementation, a combination of the $LDL^{\tran}$, LDLIQR2 and PLG approach should be considered, as it is possible for each iteration step to be performed by a different method. For badly conditioned matrices, the first steps could be performed via PLG. As soon as the condition number of the iterate has improved, another method could be employed, which shows better performance.

We now compare the developed algorithms with other available methods, in particular concerning convergence properties. A standard approach for computing (generalized) polar decompositions is the scaled Newton iteration (see e.g. \cite{Hig08}). For a given signature matrix $\Sigma$, it is given as
\begin{align}\label{Eq:ScaledNewton}
 X_{k+1} = \frac{1}{2}(\mu_k X_k + \mu_k^{-1}\Sigma X_k^{-*} \Sigma),\qquad X_0 = A.
\end{align}
It is called the Newton iteration as it represents the Newton method for solving $A^\star A= I$. See also \cite{Hig03} for details. For the DWH iteration, we have shown in Corollary \ref{Co:IterationOnS} that the iteration acts as a matrix sign function iteration on the self-adjoint factor of the decomposition. This observation also holds for the Newton iteration. Let $X_k=WS_k$ be a generalized polar decomposition of the iterate, then \eqref{Eq:ScaledNewton} is equivalent to
\begin{align*}
 X_{k+1} = W \left(\frac{1}{2} (\mu_k S_k + \mu_k^{-1}S_k^{-1})\right),\qquad X_0 = A.
\end{align*}
The part in large parentheses is the Newton iteration for the matrix sign function acting on $S_k$. In the standard setting, the self-adjoint factor is Hermitian and its eigenvalues are real. This is exploited to devise scaled iterations which drive these values closer to one and therefore accelerate convergence (see \cite{Hig08,ByeX08, NakBG10}). For the generalized polar decomposition, the values are not necessarily real. In this case, we can fall back on scaling strategies for the matrix sign function which show good convergence properties in practice. In particular, we consider determinantal scaling \cite{Bye87}, where
\begin{align*}
 \mu_k := |\det{S_k}|^{-\frac{1}{n}} = |\det{X_k}|^{-\frac{1}{n}}.
\end{align*}
The computation via the iterate $X_k$ becomes possible because signature matrices and automorphisms with respect to them have a determinant of $\pm1$. Its computation is cheap as it can be computed from the diagonal values of the LU factorization, which is used to compute $X_k^{-*}$.  For the next numerical example, we generate matrices for which the generalized polar decomposition with respect to $\Sigma$ is guaranteed to exist, but where the eigenvalues of the self-adjoint factor are all complex. 

\paragraph{Example 2}
 For the generalized polar decomposition $A=WS$, we prescribe the self-adjoint factor $S$ with a condition number $\kappa=10^k$.  The absolute values $r_j$ of the eigenvalues $\lambda_j=r_j\exp{(\iu \phi_j)}$ of $H$  are uniformly distributed between $10^{-\floor{k/2}}$ and $10^{\ceil{k/2}}$. $\phi_j$ is uniformly distributed between $-\pi/2$ and $\pi/2$, i.e.\ all eigenvalues lie in the right half plane. $S$ is generated using two random orthogonal matrices $Q_1,Q_2\in\mathbb{R}^{n\times n}$, $Q=\begin{bmatrix}Q_1&0\\0&Q_2\end{bmatrix}$, 
 \begin{align*}
S:=Q^ {\tran}
 \begin{bmatrix}
  \Real{\lambda_1}&&&-\Imag{\lambda_1}&&\\
  &\ddots&&&\ddots&\\
  &&\Real{\lambda_n}&&&-\Imag{\lambda_n}\\
    \Imag{\lambda_1}&&&\Real{\lambda_1}&&\\
  &\ddots&&&\ddots&\\
  &&\Imag{\lambda_n}&&&\Real{\lambda_n}\\  
 \end{bmatrix}Q.
 \end{align*}

 The polar factor $W$ is prescribed as 
\begin{align*}
W:=\begin{bmatrix}
    Q_3&\\&Q_4
   \end{bmatrix}\begin{bmatrix}C_W & S_W \\ S_W & C_W\end{bmatrix}.
\end{align*} 
   $Q_3$ and $Q_4$ are random orthogonal matrices. The matrix $\begin{bmatrix}C_W & S_W \\ S_W & C_W\end{bmatrix}$ describes a series of hyperbolic Givens rotations, i.e.\ 
   \begin{align*}
   C_W=\diag{\cosh{\omega_1}, \dots, \cosh{\omega_{2n}}},\quad S_W=\text{diag}(\sinh{\omega_1},\dots, \sinh{\omega_{2n}}),
   \end{align*}
   where $\omega_1,\dots,\omega_{2n}$ are uniformly distributed angles between $0$ and $\frac{1}{4}\pi$. Averaged results for 20 matrices of size $200\times 200$ ($n=100$) are given in Table \ref{Tab:DWHvsNewton}.
   \par

   For the Newton iteration we use the stopping criterion given in \cite{Hig08}, Chapter 8:
   \begin{align}
    \|X_{k++1}-X_{k}\|_F\leqq (2\epsilon)^{\frac{1}{2}},
   \end{align}
   where $\epsilon$ denotes the machine precision.

For the $\Sigma DWH$ iteration, we employ permuted graph bases (Algorithm \ref{Alg:PermGraphGenPolDec}), available in the \texttt{pgdoubling} package associated with \cite{MehP12}. It is compared to the Newton iteration with determinantal scaling (DN) and the Newton iteration with sub-optimal scaling \cite{ByeX08} (SON). We generate 20 different random matrices and report the average number of iterations and the resulting residual $\|A-\tilde{W}\tilde{S}\|_{F}/\|A\|_{F}$, where $\tilde{W}$ and $\tilde{S}$ are the computed polar factors. We influence the condition number of $A$ indirectly via $\kappa=\cond{S}$. It is about twice as high as $\kappa$ because of the used hyperbolic Givens rotations.

\begin{table}[t]
\centering
\caption{Convergence behavior for different methods computing the generalized polar decomposition with respect to $\Sigma$ of a $200\times 200$ matrix (Example 2).\label{Tab:DWHvsNewton}}
\begin{tabular}{lllllll}
\toprule
&$\kappa$&$10$&$10^{5}$&$10^{10}$&$10^{15}$\\
&\texttt{cond}(A)
&2.15e+01&1.98e+05&1.98e+10&2.02e+15\\

\midrule
\multirow{3}{*}{\# iterations}
&$\Sigma$DWH   &8.70&9.70&10.65&10.60   \\
&DN	   	       &12.30&20.00&32.95&44.13 \footnotemark[1]\\
&SON	         &14.05&15.45&16.45&16.74 \footnotemark[2]\\ \midrule
\multirow{3}{*}{residual}
&$\Sigma$DWH			&5.06e-15&7.68e-15&9.88e-15&3.00e-15\\ 
&DN	 			        &2.98e-15&2.98e-15&2.93e-15&2.96e-15\\      
&SON	   	      	&3.00e-15&2.98e-15&2.96e-15&2.89e-15\\ \midrule
\multirow{3}{*}{rel. error $W$}
&$\Sigma$DWH			&1.35e-14&9.45e-12&5.35e-08&8.01e-03\\ 
&DN	 		        	&1.18e-14&3.96e-11&1.53e-06&7.83e-02\\      
&SON	   		     	&1.32e-14&3.12e-11&2.24e-07&5.22e-03\\   \midrule
\multirow{3}{*}{rel. error $S$}
&$\Sigma$DWH			&1.05e-14&2.76e-14&3.51e-14&4.51e-14\\ 
&DN	 		        	&9.98e-15&9.00e-12&8.52e-07&6.65e-02\\      
&SON	          	&1.43e-14&2.42e-14&2.33e-14&2.83e-14\\   \midrule
\multirow{3}{*}{$\|\Sigma W^{\tran}\Sigma W - I\|_{F}$}
&$\Sigma$DWH			&1.16e-15&1.23e-15&1.07e-15&1.25e-15\\ 
&DN	 			        &3.19e-15&3.19e-15&3.13e-15&3.12e-15\\      
&SON	   		     	&3.21e-15&3.18e-15&3.16e-15&3.09e-15\\  
\bottomrule

\end{tabular}
\footnotesize
\vspace{1ex}
       \\\footnotemark[1] 5 out of 20 runs did not converge. $\qquad$ \footnotemark[2] 1 out of 20 runs did not converge.
\end{table}

In the standard setting, DWH and SON converge in 6 \cite{NakBG10}, respectively 9 \cite{ByeX08}, steps. Here, the iterations act as scalar iterations on the eigenvalues of the self-adjoint factor, who happen to be real in the standard case, but not in the indefinite setting. Still we can observe that they converge significantly faster than the Newton iteration with determinantal scaling, in particular for ill-conditioned matrices.  $\Sigma$DWH generally seems to need about 2/3 as many iteration steps as SON. Whether the cost per iteration is comparable, depends on the chosen implementation method for the DWH iteration. The simplest method is based on one $LDL^{\tran}$ decomposition \eqref{Eq:LDLDWH} and the main cost is a symmetric matrix inversion, just as in the Newton variants. If higher stability is needed in the case of badly conditioned matrices, it can be obtained at the expense of a higher costs per iteration. This can be done by employing the LDLIQR2 iteration or by improving the corresponding subspace via Lagrangian graph bases (Algorithm \ref{Alg:PermGraphGenPolDec}).  

$\Sigma$DWH displays the lowest backward error for the $\Sigma$-orthogonal factor W, which deteriorates for all methods as matrices become ill-conditioned. All methods yield a factor $W$ that shows a good $\Sigma$-orthogonality. SON and $\Sigma$DWH both do a much better job than DN at recovering the self-adjoint factor $S$ with backward errors of order $10^{-14}$ instead of $10^{-2}$.  DN and SON sometimes fail to converge for badly conditioned matrices. 


We see that $\Sigma$DWH can compete with standard methods, even if no definite pseudosymmetric structure is given. Note that $\Sigma$DWH is the only one of the three methods that can directly be applied to non-square matrices, in order to compute the canonical generalized polar decomposition. 

The results of Example 2 should be seen as preliminary, as the scaling factors and the stopping criterion \eqref{Eq:StopCrit} are not completely justified in the non-definite case. They do, however, motivate further research to devise iterations based on rational functions acting on complex values.

\paragraph{Example 3}
We generate pseudosymmetric matrices as in Example 1, but additionally ensure the definiteness of $\Sigma A$ by choosing only positive values for $D$. We compare the same methods as in Example 2 with respect to convergence properties. 20 matrices were generated and averaged results are reported in Table \ref{Tab:DWHvsNewtonDefinite}.\par

\begin{table}[t]
\centering
       \caption{Convergence behavior for different methods computing the generalized polar decomposition with respect to $\Sigma$ of definite pseudosymmetric matrices of size $200\times 200$ (Example 3).\label{Tab:DWHvsNewtonDefinite}}
\begin{tabular}{lllllll}
\toprule
&$\kappa$&$10$&$10^{5}$&$10^{10}$&$10^{15}$\\
\midrule
\multirow{3}{*}{\# iterations}
&$\Sigma$DWH&4.00&5.00&6.00&6.00\\
&DN	   			&6.00&15.10&30.50&44.50\\
&SON	    	&6.00&7.00&8.00&9.00\\				\midrule
\multirow{3}{*}{residual}
&$\Sigma$DWH   &1.38e-15&4.47e-14&2.34e-14&2.85e-14\\ 
&DN	 		       &8.11e-16&2.46e-14&5.30e-14&1.05e-14\\ 
&SON	         &8.14e-16&3.20e-14&3.03e-14&1.04e-14\\ \midrule
\multirow{3}{*}{$\|\Sigma W^{\tran}\Sigma W - I\|_{F}$}
&$\Sigma$DWH			&1.26e-15&1.95e-13&2.03e-13&6.92e-14\\ 
&DN	 		        	&7.31e-16&6.87e-14&5.66e-14&3.13e-14\\      
&SON	   		     	&7.16e-16&6.94e-14&5.64e-14&3.09e-14\\  
\bottomrule
\end{tabular}
\end{table}

As expected, we see the convergence of $\Sigma$DWH and of the Newton iteration with suboptimal scaling within 6, respectively 9, iterations. 

\section{Conclusions}\label{Sec:Conclusions}
In this paper, we have presented a generalization of the QDWH method to compute the canonical generalized polar decomposition of a matrix with respect to a signature matrix $\Sigma$. If $\Sigma$ is chosen as the identity, the hyperbolic QR decomposition becomes the standard QR decomposition and can safely be computed with the column elimination approach. This yields the well-known QDWH iteration. 

Several options were provided on how to realize the iterations. While the column elimination based hyperbolic QR decomposition forms the most natural generalization of QDWH, it does not yield the best results regarding stability. LDLIQR2 (Section \ref{Sec:LDLHQR2}) or employing permuted (Lagrangian) graph bases (Algorithm \ref{Alg:PermGraphGenPolDec} and \ref{Alg:PermLangGraphGenPolDec}) perform better in this regard.

Using these variants, a stability similar to Newton methods can be observed, but fewer iterations are needed. For the important class of definite pseudosymmetric matrices, the convergence behavior corresponds to the standard QDWH method. Convergence up to machine precision can be guaranteed in 6 steps for reasonably conditioned matrices.

The theoretical results we gave, in particular Lemma \ref{Lem:GenQRIter}, provide a greater flexibility in the algorithmic design for DWH-based iterations, which might be utilized further than the scope of this paper permits. Other methods for computing well-conditioned bases could also yield good results. Being more flexible in algorithmic design becomes increasingly important in view of modern computer architectures. In general these become more heterogeneous. They employ different levels of parallelism on various scales and have restrictions on available memory or use numerous accelerators and GPUs. Our framework provides the flexibility to find solutions, which could exploit the architecture at hand to its full potential. 

Our main motivation came from computing the matrix sign function of large definite pseudosymmetric matrices. Here, the iteration acts as a rational function on what can be understood as generalized singular values. Hence, further developments using ideas from \cite{NakF16} are possible. Using Zolotarev's functions as best-approximations to the sign function of higher degree, yields an iteration that converges in two steps. The individual steps take more work but are embarrassingly parallel and well-suited for large-scale high performance computations. In the field of computational quantum physics this is exactly what is needed making this research direction promising. 

Computing the hyperbolic QR decomposition is useful in many applications, which could benefit from the analysis given in Section \ref{Sec:LDLHQR}. In particular the LDLIQR2 method (Algorithm \ref{Alg:LDLIQR2}) is a promising technique to tackle problems associated with the stability of the hyperbolic or indefinte QR decomposition.
\addcontentsline{toc}{section}{References}


\end{document}